\documentclass[10pt]{article}

\usepackage{moreverb,url}
\usepackage{enumitem}
\usepackage{amssymb}
\usepackage{amsfonts,amsmath,mathtools,amsthm}
\usepackage{mathrsfs}
\usepackage{mathtools}

\mathtoolsset{showonlyrefs}

\usepackage[colorlinks,bookmarksopen,bookmarksnumbered,citecolor=red,urlcolor=red]{hyperref}

\newcommand\BibTeX{{\rmfamily B\kern-.05em \textsc{i\kern-.025em b}\kern-.08em
T\kern-.1667em\lower.7ex\hbox{E}\kern-.125emX}}

\usepackage{setspace}

\newcommand{\diff}[2]{\frac{\mathrm{d} #1}{\mathrm{d} #2}} 			
\newcommand{\bdot}{:}										
\DeclareMathOperator{\dist}{dist}								
\DeclareMathOperator{\dom}{dom}								
\DeclareMathOperator{\spt}{spt}								

\newcommand{\rbr}[1]{\left( #1 \right)} 							
\newcommand{\sbr}[1]{\left[ #1 \right]} 							
\newcommand{\cbr}[1]{\left\{ #1 \right\}} 							
\newcommand{\fixed}[2]{\left. #1 \right|_{#2}} 						
\newcommand{\norm}[1]{\left\lVert#1\right\rVert}	 				
\newcommand{\abs}[1]{\left\lvert #1 \right\rvert} 					
\newcommand{\ipbr}[1]{\left\lfloor #1 \right\rfloor} 							
\newcommand{\txtsub}[1]{_{\text{#1}}} 							
\newcommand{\txtsup}[1]{^{\text{#1}}} 							
\newcommand{\dmn}{\Omega}									
\newcommand{\inprod}[2]{\rbr{ #1, #2 }} 							
\newcommand{\inprodalt}[2]{\langle #1, #2 \rangle} 					
\newcommand{\subeps}{_{\varepsilon}}
\newcommand{\subepsh}{_{\varepsilon_{h}}}
\newcommand{\openset}{A}
\newcommand{\opensetalt}{B}
\newcommand{\prfGamma}{\Gamma\text{-}}	
	
\newcommand{\meas}{\mu}

\newcommand{\dmndim}{n}									
\newcommand{\rngdim}{n}									
\newcommand{\rngdimalt}{m}									
\newcommand{\symspace}{\mathbb{M}\txtsub{sym}^{\rngdim \times \dmndim}}

\newcommand{\mtxspace}{\mathbb{M}^{\rngdim \times \dmndim}}
\newcommand{\mtxspacealt}{\mathbb{M}^{\rngdimalt \times \dmndim}}
\newcommand{\mtx}{\xi}
\newcommand{\mtxalt}{\eta}
\newcommand{\p}{^{+}}
\newcommand{\m}{^{-}}
\newcommand{\wnabla}{\mathrm{D}}
\newcommand{\wstn}{\mathrm{E}}
\newcommand{\symsup}{\txtsup{s}}
\newcommand{\microgeom}{B}
\newcommand{\cone}{K}
\newcommand{\mcone}{H}
\newcommand{\mden}{g}
\newcommand{\mfun}{G}
\newcommand{\cden}{f}

\newcommand{\cfun}{F}
\newcommand{\sing}{\txtsup{s}}
\newcommand{\jumpp}{\txtsup{j}}
\newcommand{\cantor}{\txtsup{c}}
\renewcommand{\hom}{\txtsub{hom}}
\newcommand{\intvol}[2]{\int_{#1}  #2 \mathrm{d}x} 					
\newcommand{\intsurf}[2]{\int_{#1}  #2 \mathrm{d}\cH^{\dmndim-1}} 		
\newcommand{\at}[1]{\!\rbr{#1}}						 			
\newcommand{\R}[1]{\mathbb{R}^{#1}}
\newcommand{\Z}[1]{\mathbb{Z}^{#1}}
\newcommand{\N}{\mathbb{N}}
\newcommand{\barR}{\overline{\mathbb{R}}}

\newcommand{\cA}{\mathcal{A}}

\newcommand{\cE}{\mathcal{E}}
\newcommand{\cH}{\mathcal{H}}
\newcommand{\cL}{\mathcal{L}}
\newcommand{\cM}{\mathcal{M}}
\newcommand{\cU}{\mathcal{U}}

\newcommand{\mres}{\,\mathbin{\vrule height 1.6ex depth 0pt width
0.13ex\vrule height 0.13ex depth 0pt width 1.3ex}}

\setlist{nosep}

\newtheorem{theorem}{Theorem}[section] 		
\newtheorem{proposition}[theorem]{Proposition}	
\newtheorem{lemma}[theorem]{Lemma}
\newtheorem{corollary}[theorem]{Corollary}

\numberwithin{equation}{section}

\begin{document}


\title{Homogenization of cohesive fracture \\ in masonry structures}

\author{Andrea Braides\\
Department of Mathematics\\
University of Rome Tor Vergata,\\
via della Ricerca Scientifica 1, 00133 Rome, Italy\\
\\
Nicola A. Nodargi\\
 Department of Civil Engineering and Computer Science\\
 University of Rome Tor Vergata\\
 via del Politecnico 1, 00133 Rome, Italy
 }

\date{}




\maketitle

\begin{abstract}
We derive a homogenized mechanical model of a masonry-type structure constituted by a periodic assemblage of blocks with interposed mortar joints. The energy functionals in the model under investigation consist in (i) a linear elastic contribution within the blocks, (ii) a Barenblatt's cohesive contribution at contact surfaces between blocks and (iii) a suitable unilateral condition on the strain across contact surfaces, and are governed by a small parameter representing the typical ratio between the length of the blocks and the dimension of the structure. Using the terminology of $\Gamma$-convergence and within the functional setting supplied by the functions of bounded deformation, we analyze the asymptotic behavior of such energy functionals when the parameter tends to zero, and derive a simple homogenization formula for the limit energy. Furthermore, we highlight the main mathematical and mechanical properties of the homogenized energy, including its non-standard growth conditions under tension or compression. The key point in the limit process is the definition of macroscopic tensile and compressive stresses, which are determined by the unilateral conditions on contact surfaces and the geometry of the blocks.
\end{abstract}

\section{Introduction}
In this work we examine the mechanical modeling of a class of materials referred to as \emph{masonry-like materials}. Such materials are characterized by different behaviors in tension or compression, possibly undergoing fracture. A noteworthy instance in that class is represented by dry masonry, typical of historical buildings, consisting in an assemblage of blocks that are in unilateral contact with each other. The fact that the blocks can be detached at no energy expense results in a fully degenerate overall behavior, with vanishing resistance of the material under tension and the possibility of an unresisted cracking process. Alternative to dry masonry, in building practice it is widespread to interpose mortar joints between bricks. In that case, blocks are still in contact with each other, but a (small) resistance to their detachment is offered by the mortar, which is responsible for cohesive tractions resisting the cracking process. 

A first general mathematical treatment of the dry-masonry problem goes back to the works by Giaquinta and Giusti \cite{Giaquinta_Giusti_ARMA_1985} and Anzellotti \cite{Anzellotti_AIHP_1985}. In those works, in the framework of infinitesimal strain theory and under the assumption of a linear elastic behavior, energy of deformations of the form
\begin{equation}
	\mfun\at{u} = \intvol{\dmn}{\mden\at{P_{\cone^{\perp}}\cE u}}
\label{eq:intro_masonry}
\end{equation} 
are considered, where $\dmn \subset \R\dmndim$ is a reference configuration domain, $\mden$ is a linear elastic energy density and $\cE u$ is the strain associated to the displacement field $u$. The degeneracy of the material under tension is modeled by the introduction of a {\em cone $\cone$ of tensile strains}, such that only the projection $P_{\cone^{\perp}} \cE u$ of the strain onto the dual {\em cone $\cone^\perp$ of compressive strains} determines energy storage. This degenerate behavior implies that the energy of deformation $\mfun$ is not coercive on any reasonable normed space of admissible displacements. However, by replacing the energy of deformation with the total potential energy (i.e.~considering also the contribution of the work of the external forces) and introducing suitable safety conditions (indeed prescribing the external forces to be compressive), coerciveness can be recovered in the space $BD\at\dmn$ of \emph{functions of bounded deformation} on~$\dmn$; that is, the space of functions whose distributional strain $\wstn u$ is a measure \cite{Temam_1985}. Following the direct method of the Calculus of Variations, existence theorems are derived by showing that the total potential energy fulfills a lower-semicontinuity property on $BD\at\dmn$ as well. A mechanical insight in the subject can be found e.g.~in~\cite{DelPiero_1998, Nodargi_Bisegna_2019}.

As an alternative to that  \emph{macroscopic} description of a masonry material, a \emph{microscopic} approach has been adopted in \cite{Braides_ChiadoPiat_2000}. In that approach, the assemblage of blocks interacting through their contact surfaces is regarded as the periodic microstructure of a material whose macroscopic properties are determined by \emph{homogenization}; i.e.,~by analyzing its asymptotic behavior as the characteristic size of the microstructure vanishes. In mathematical terms, a $\rbr{\dmndim-1}$-dimensional periodic closed set $\microgeom$ is introduced, to be scaled by the characteristic size $\varepsilon$ of the microstructure, and the reference configuration $\dmn$ is subdivided into a periodic collection of disconnected sets~$\dmn \setminus \varepsilon\microgeom$. A space of admissible displacements is then considered as a prescription for the block kinematics:
\begin{equation}
	\cU\subeps\at\dmn = \cbr{u \in SBD\at{\dmn} \bdot \,\,J_{u} \subseteq \varepsilon\microgeom, \,\, \rbr{u\p - u\m} \odot \nu_{u} \in \cone_0\,\, \cH^{n-1}\text{-a.e}}, 
\label{eq:intro_adm_displ}
\end{equation}
where $SBD\at\dmn$ is the space of \emph{special functions of bounded deformation} on $\dmn$ (e.g., see \cite{Ambrosio_DalMaso_ARMA_1997}), $J_{u}$ is the jump set of $u$,  $u^{\pm}$ are the traces of $u$ on both sides of $J_{u}$ and $\nu_{u}$ is the unit normal to $J_{u}$. This analytical description translates the fact that the admissible displacements are 
functions whose jumps localize at the interfaces of the blocks and fulfill a unilateral condition prescribed by the cone of matrices $\cone_{0}$. Possible choices for that cone include~$\cone_{0} = \cbr{a \odot b \bdot b = \lambda a, \lambda \geq 0}$ or $\cone_{0} = \cbr{a \odot b \bdot \inprod{a}{b} \geq 0}$, respectively implying an infinitesimal no-sliding condition (i.e., infinite friction assumption) and a detachment condition on the opening of a crack (i.e., vanishing friction assumption). Upon noticing that admissible displacements $u$ are such that $u \in H^1\at{\dmn\setminus\varepsilon\microgeom}$, and hence the strain $\wstn u$ reduces to its absolutely continuous part $\cE u$ with respect to the Lebesgue measure on $\dmn\setminus\varepsilon\microgeom$, the following microscopic linearly elastic energy is considered in~\cite{Braides_ChiadoPiat_2000}: 
\begin{equation}
	\mfun\subeps\at{u} = \frac{1}{2}\intvol{\dmn\setminus\varepsilon\microgeom}{\inprod{A \cE u}{\cE u}}, \quad
	u \in \cU\subeps\at\dmn,
\end{equation}
with $A$ as a fixed fourth-order tensor. A \emph{homogenized energy density} $\mden\hom$ can be defined from $A$ and $B$ by a homogenization formula optimizing over periodic perturbations of a given strain. The function $\mden\hom$ may vanish on a  set of matrices, which we denote as $\cone\hom$, the cone of the \emph{homogenized tensile strains}. Under the assumption that $\mden\hom\at\cdot = \mden\hom\at{P_{\cone\hom^{\perp}}\cdot}$ (which is satisfied in usual examples), it is shown that 
this microscopic approach leads to a masonry-type energy; more precisely, that the family $\rbr{\mfun\subeps}$
$\prfGamma$converges as $\varepsilon \rightarrow 0\p$ to a homogenized energy $\mfun\hom$ on $BD\at\dmn$ of the form
\begin{equation}
	\mfun\hom\at{u} = \intvol{\dmn}{\mden\hom\at{P_{\cone\hom^{\perp}}\cE u}}, \quad
	u \in \cU\hom\at\dmn.
\label{eq:intro_masonry_hom}
\end{equation} The space of \emph{homogenized admissible displacements} $\cU\hom\at\dmn$ is the set of those functions $u \in BD\at\dmn$ such that the projection $P_{\cone\hom^{\perp}}\wstn\sing u$, with $\wstn\sing u$ as the singular part of the strain measure $\wstn u$ and $\cone\hom^{\perp}$ as the cone orthogonal to $\cone\hom$, vanishes. 

As regards the problem of masonry with mortar joints between bricks, many contributions from the mechanical literature could be mentioned, including: macromechanical continuum models based on phenomenological constitutive laws (e.g., \cite{Pela_Roca_CBM_2013, Girardi_Pellegrini_MMS_2018}), micromechanical models (e.g., \cite{Gambarotta_Lagomarsino_EESD_1997, Oliveira_Lourenco_CS_2004}), and homogenized multiscale models (e.g., \cite{Milani_CS_2011, Marfia_Sacco_CMAME_2012, Addessi_Sacco_IJSS_2012}). However, a rigorous mathematical treatment seems to be missing in the framework of homogenization theory for periodic masonry with mortar joints on the space of functions of bounded deformation, i.e.~explicitly considering fractures as stemming from discontinuous displacement fields.

To develop such a homogenization theory, the approach discussed in~\cite{Braides_ChiadoPiat_2000} can be taken as departing point. In particular, accounting for the mortar requires to include some surface energy contribution on the discontinuity set $\varepsilon\microgeom$ at the microscopic level. Here we focus on Barenblatt's model of cohesive fracture, prescribing an isotropic surface energy, positively homogeneous of degree one \cite{Barenblatt_1962}. Accordingly, we are led to considering the family of functionals $\rbr{\cfun\subeps}$ given by
\begin{equation}
	\cfun\subeps\at{u} = \frac{1}{2}\intvol{\dmn\setminus\varepsilon\microgeom}{\inprod{A\cE u}{\cE u}}
		+\intsurf{J_{u}\cap\varepsilon\microgeom}{\abs{u\p-u\m}}, \quad
		u \in \cU\subeps\at\dmn,
\end{equation}
with the space of admissible displacements~$\cU\subeps\at\dmn$ given in~\eqref{eq:intro_adm_displ}.
The main result of the present work consists in proving that the functionals $\rbr{\cfun\subeps}$ $\prfGamma$converge as $\varepsilon \rightarrow 0\p$ to a homogenized energy $\cfun\hom$ on $BD\at\dmn$ given by
\begin{equation}
	\cfun\hom\at{u} = \intvol{\dmn}{\cden\hom\at{\cE u}}
			+\intvol{\dmn}{\cden\hom^{\infty}\at{\diff{\wstn\sing u}{\abs{\wstn\sing u}}}\mathrm{d}\abs{\wstn\sing u}}, \quad
			u \in \cU\hom\at\dmn,
\label{eq:intro_cohesive}
\end{equation}
where $\cden\hom$ is a \emph{cohesive homogenized energy density}, $\cden\hom^{\infty}$ is the recession function of $\cden\hom$, $\wstn\sing u$ denotes the singular part of the strain $\wstn u$ with respect to the Lebesgue measure and the space of homogenized admissible displacements $\cU\hom\at\dmn$ is the same as that appearing in~\eqref{eq:intro_masonry_hom} in the case of dry masonry. Note that in this case no projection structure can be obtained in the limit energy. This allows to remove the assumptions on the cohesive homogenized energy density $\cden\hom$ required for the homogenization result in~\cite{Braides_ChiadoPiat_2000}.
We give an explicit homogenization formula for the energy density $\cden\hom$, which depends on both the microgeometry $\microgeom$ and the cone $\cone_0$. It is noteworthy to observe that such energy density satisfies a non-standard growth condition
\begin{equation}
	c_1 \rbr{\abs{\mtx} - 1} \leq \cden\hom\at\mtx \leq c_2 \abs{\mtx}^2,
\end{equation}
so that the present homogenization theorem does not fit in the framework of any of the general integral representation results as in \cite{Buttazzo_1989, DalMaso_1993}. As for the growth condition from below, that condition cannot be improved. In fact, we prove that the homogenized energy density $\cden\hom$ has sublinear growth over the cone $\cone\hom$ of homogenized tensile stresses as a natural mechanical consequence of the surface Barenblatt's energy contribution at microscopic level. Furthermore, the growth condition from below allows the homogenized energy~$\cfun\hom$ to be regarded as an instance of demi-coercive functionals, as introduced in \cite{Anzellotti_DalMaso_NATMA_1986}, thus implying an existence theorem for the related minimization problem. 

Energies of the form~\eqref{eq:intro_cohesive} have been broadly investigated as the relaxation in the $L^1$-topology of functionals defined on $SBD\at\dmn$ and involving interaction between bulk and surfaces energies (e.g., see \cite{Braides_Coscia_PRSE_1994, Braides_Vitali_AMO_1997, Braides_Vitali_ASNSP_2001}). In particular, connections can be found with the result discussed in~\cite{Braides_Vitali_ASNSP_2001}, where the relaxation of elastic energies with unilateral constraints on the strains is considered. Some contact points can also be recognized with homogenization results discussed in~\cite{Francfort_Giacomini_JEMS_2014, Ansini_ChiadoPiat_BUMI_1999, Stelzig_ESAIM_2012}. Specifically, in~\cite{Francfort_Giacomini_JEMS_2014}, the problem of periodic homogenization in perfect elasto-plasticity is discussed. In~\cite{Ansini_ChiadoPiat_BUMI_1999} the homogenization of integral functionals involving energies concentrated on periodic multidimensional structures and defined on Sobolev spaces with respect to measures is considered. In~\cite{Stelzig_ESAIM_2012}, the homogenization of many-body structures undergoing large displacements and obeying a non-interpenetration constraint is dealt with. As a major difference, the large-displacement framework calls for a functional setting in the space $BV$ (instead of $BD$) and the non-interpenetration constraint translates into a global condition (instead of a local condition governed by the cone $\cone_{0}$). To the best of the authors' knowledge, functionals of the form~\eqref{eq:intro_cohesive} are new in the context of homogenization theory for masonry structures on $BD$. 

The present work is organized as follows. In Section~\ref{s:notation} we set the main notation and collect some definitions and well-known results needed in the subsequent developments. We discuss the problem statement and the main result in Section~\ref{s:main_result}. Section~\ref{s:1d_problem} is devoted to an illustration of the main result in a one-dimensional setting. That discussion is instrumental as it highlights some properties of the homogenized energy density which are explored in the general case in Section~\ref{s:properties}. We give the proof of the main result in the technical Section~\ref{s:proof}. Conclusions and perspectives are outlined in Section~\ref{s:conclusions}.

\section{Notation}\label{s:notation}

We denote by $\inprod{\cdot}{\cdot}$ and $\abs{\cdot}$ the scalar product and the induced norm in $\R\dmndim$, for any $\dmndim \geq 1$. Upon identification with the space $\R{\rngdimalt \dmndim}$, the same notation is also adopted for the vector space $\mtxspacealt$ of $\rngdimalt \times \dmndim$ real matrices. The symbol $\symspace$ is used for the subspace of  symmetric matrices in $\mtxspace$. In particular, for $\mtx \in \mtxspace$, we denote by $\mtx\txtsup{s}$ its symmetric part. Given $\mtx \in \mtxspace$, we write $u_{\mtx}$ for the linear function $u_{\mtx}\at{x} = \mtx{x}$.
Given $a, \, b \in \R{\dmndim}$, the tensor product $a \otimes b$ is the $\dmndim \times \dmndim$ matrix with entries $a_ib_j$ for $i, \, j = 1, \dots{}, \dmndim$. The symmetric tensor product is defined as $a \odot b = \rbr{a\otimes b + b\otimes a}/2$. Note  that $\abs{a \odot b}^2 = \rbr{\abs{a}^2\abs{b}^2 + \inprod{a}{b}^2}/2$. 

The open ball in $\R{\dmndim}$ with radius $\rho$ and center $x$ is denoted by $B_{\rho}\at{x}$. If $\dmn$ is an open set in $\R{\dmndim}$, we denote by $\cA\at\dmn$ the class of open subsets of $\dmn$ and by $\cA_{0}\at\dmn$ the class of bounded open subsets of $\dmn$. Moreover, $L^2\at{\dmn; \R{\rngdimalt}}$ and $H^1\at{\dmn; \R{\rngdimalt}}$ stand for the usual Lebesgue and Sobolev spaces of $\R\rngdimalt$-valued functions. If $\rngdimalt = 1$, we simply write $L^2\at{\dmn}$ and $H^1\at{\dmn}$. The symbol $C^{k}_0\at{\dmn; \R{\rngdimalt}}$, $0 \leq k \leq \infty$ refers to the space of compactly supported smooth functions. We denote by $\cM\at{\dmn; \mtxspace}$ the space of $\mtxspace$-valued Borel measures and by $\cH^{\dmndim-1}$ the $\rbr{\dmndim-1}$-dimensional Haussdorff measure in $\R{\dmndim}$.

For $f \bdot \R\rngdimalt \rightarrow [0, +\infty)$ a convex function and $\meas \in \cM\at{\openset; \mtxspace}$ a measure, we use the notation  $\int_{\openset}{f\at\meas} = \intvol{\openset}{f \rbr{h}} + \int_{\openset}{f^{\infty}\at{\diff{\meas\txtsup{s}}{\abs{\meas\sing}}} \mathrm{d}\!\abs{\meas\sing}}$, where $\meas = h\cL^{\dmndim} + \meas\sing$ is the Lebesgue decomposition of $\meas$ and $f^{\infty}$ is the recession function of $f$ (e.g., see \cite{Rockafellar_1970}) defined by
$$
f^{\infty}(\xi)=\lim_{t\to+\infty} {f(t\xi)\over t}.
$$



\subsection{The space $BD$} 
The space $BD\at\dmn$ of \emph{functions of bounded deformation} is the space of the functions $u \in L^1\at{\dmn; \R\rngdim}$ whose symmetric distributional gradient 
\begin{equation}
	\wstn u = \frac{1}{2}\rbr{\wnabla u + \wnabla u\txtsup{T}}
\end{equation} 
is a measure in $\cM\at{\dmn,\symspace}$. For any $u \in BD\at{\dmn}$, we consider the Radon-Nikodym decomposition of the strain 
\begin{equation}
	\wstn u = \cE u \,\cL^{\dmndim} + \wstn\sing u,
\end{equation}
where $\cE u$ is the density of the absolutely continuous part of $\wstn u$ with respect to $\cL^\dmndim$ and $\wstn\sing u$ is the singular part of $\wstn u$ with respect to $\cL^\rngdim$. We further decompose $\wstn \sing u$ into a jump part $\wstn \jumpp u = \wstn u \mres J_{u}$ and a Cantor part $\wstn \cantor u = \wstn\sing u \mres \rbr{\dmn \setminus J_{u}}$. Spefically, it results that 
\begin{equation}\normalfont
	\wstn\jumpp u = \rbr{u\p - u\m} \odot \nu_{u} \cH^{\rngdim-1}\mres J_{u}
\end{equation}
and $\normalfont \abs{\wstn\cantor u} \at\opensetalt = 0$ on any Borel subset $\opensetalt \subseteq \dmn$ which is $\sigma$-finite with respect to $\cH^{\rngdim-1}$. The space $SBD\at\dmn$ of \emph{special functions of bounded deformation} in $\dmn$ is the space of the functions $u \in BD\at\dmn$ such that $\normalfont \wstn\cantor u$ is the null measure; i.e., satisfying
\begin{equation}
	\wstn u = \cE u \,\cL^{\dmndim} + \rbr{u\p - u\m} \odot \nu_{u} \cH^{\rngdim-1}\mres J_{u}.
\end{equation}

A sequence $\rbr{u_h}$ in $BD\at{\dmn}$ weakly converges to a function $u \in BD\at{\dmn}$ if
$u_h \rightarrow u $ in $L^1\at{\dmn; \R\rngdim}$ and $\rbr{\abs{\wstn u_h}\at\dmn}$ is bounded.

\subsection{$\prfGamma$convergence}
Let $\rbr{X, d}$ be a metric space and $\rbr{F_h}$ be a sequence of functionals from $X$ into $\barR$. We say that $\rbr{F_h}$ $\prfGamma$converges to $F$ in $X$ with respect to the topology induced by $d$ if the following conditions are satisfied:
\begin{enumerate}[label={(\roman*)}]
	\item for every $x \in X$ and for every sequence $\rbr{x_h}$ converging to $x$ in $X$ we have $F\at{x} \leq \liminf_h F_h\at{x_h}$;
	\item for every $x \in X$ there exists a sequence $\rbr{x_h}$ converging to $x$ in $X$ such that $F\at{x} = \lim_h F_h\at{x_h}$.
\end{enumerate}
Under appropriate coercivity conditions, $\prfGamma$convergence guarantees the convergence of the minimum values of the functionals~$F_{h}$ to the minimum value of their $\prfGamma$limit.

\section{Setting of the problem and main result}\label{s:main_result}
Let $\dmn$ be a bounded open set of $\R\dmndim$. For $Y = \rbr{0,1}^\dmndim$ the unit cube of~$\R\dmndim$, we consider a closed rectifiable $Y$-periodic $\rbr{\dmndim-1}$-dimensional subset of $\R\dmndim$, i.e.~such that $\microgeom + k = \microgeom$ for all $k \in \Z{\rngdim}$. Accordingly, $\dmn$ is partitioned into a periodic collection of disconnected sets (Figure~\ref{fig:geometry}(a)). 

We fix a positive definite symmetric linear operator $A \bdot \symspace \rightarrow \symspace$, the scalar product on $\symspace$ given by  $\inprodalt{\mtx}{\mtxalt} = \inprod{A \mtx}{\mtxalt}$ and introduce the associated norm $\norm{\mtx} = \inprodalt{\mtx}{\mtx}^{1/2}$. Note that $\norm{\cdot}$ is equivalent to $\abs{\cdot}$ in $\symspace$; that is,
		\begin{equation}
		\sqrt{\alpha} \abs{\mtx} \leq \norm{\mtx} \leq M \abs{\mtx},
			\label{eq:doppia_stima}
		\end{equation}
	for suitable constants $\alpha, \, M > 0$.

We assume that $\cone_0$ is a closed cone in $\symspace$ consisting of matrices of the form $a \odot b$, satisfying the following convexity assumption:
		\begin{equation}
			a \odot \rbr{b + c} \in \cone_0 \text{ whenever } a \odot b, \, a \odot c \in \cone_0.
		\label{eq:cone_0_hp}
		\end{equation}

For $\varepsilon > 0$, we define the functionals $\cfun\subeps \bdot L^2\at{\dmn; \R\rngdim} \rightarrow \sbr{0, +\infty}$ as
\begin{equation}
	\cfun\subeps\at{u} = \begin{dcases}
		\frac{1}{2}\intvol{\dmn \setminus \varepsilon \microgeom}{\norm{\cE u}^2}
		+ \intsurf{J_u \cap \varepsilon\microgeom}{\abs{u\p-u\m}}, & u \in \cU\subeps\at{\dmn}, \\
		+\infty, \quad &\text{otherwise,}
	\end{dcases}
\label{eq:cohesive_functional}
\end{equation}
where the set $\cU\subeps\at{\dmn}$ of admissible displacements is given by
		\begin{equation}
			\cU\subeps\at\dmn = \cbr{u \in SBD\at{\dmn} \bdot \, 
				J_{u} \subseteq \varepsilon \microgeom, 
				\rbr{u\p-u\m} \odot \nu_u \in \cone\txtsub{0}  \,\, \cH^{n-1}\text{-a.e.}},
		\label{eq:test_displ}
		\end{equation} 
	i.e.~consists of all special functions with bounded deformation whose jump set is contained in~$\varepsilon \microgeom$ and such that the density of the singular part of the strain belongs to the cone~$\cone\txtsub{0}$ (Figure~\ref{fig:geometry}(b)), which describes the admissible singular part of the strain. 

Accordingly, over the set of admissible displacements, the material behaves as a linear elastic medium within brick regions $\dmn\setminus\varepsilon\microgeom$, and a surface energy over interfaces between adjacent bricks, following the Barenblatt's model of cohesive fracture  \cite{Barenblatt_1962}, is considered. Possible fractures localize over such interfaces and produce displacement jumps obeying the unilateral constraint associated to $\cone_0$. 
\begin{figure}
	\centering
	\includegraphics[clip = true, scale = 0.3]{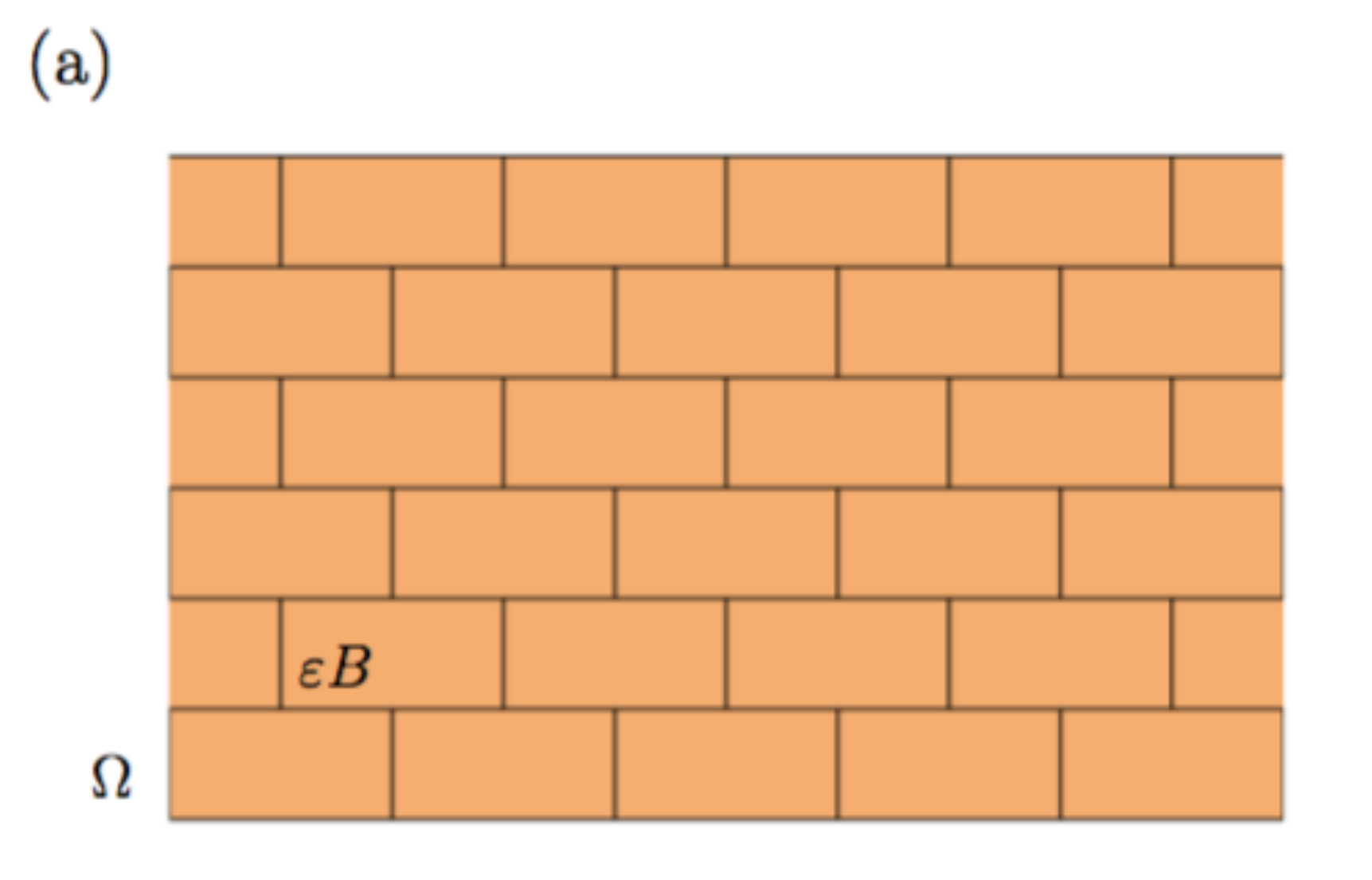}
	\includegraphics[clip = true, scale = 0.3]{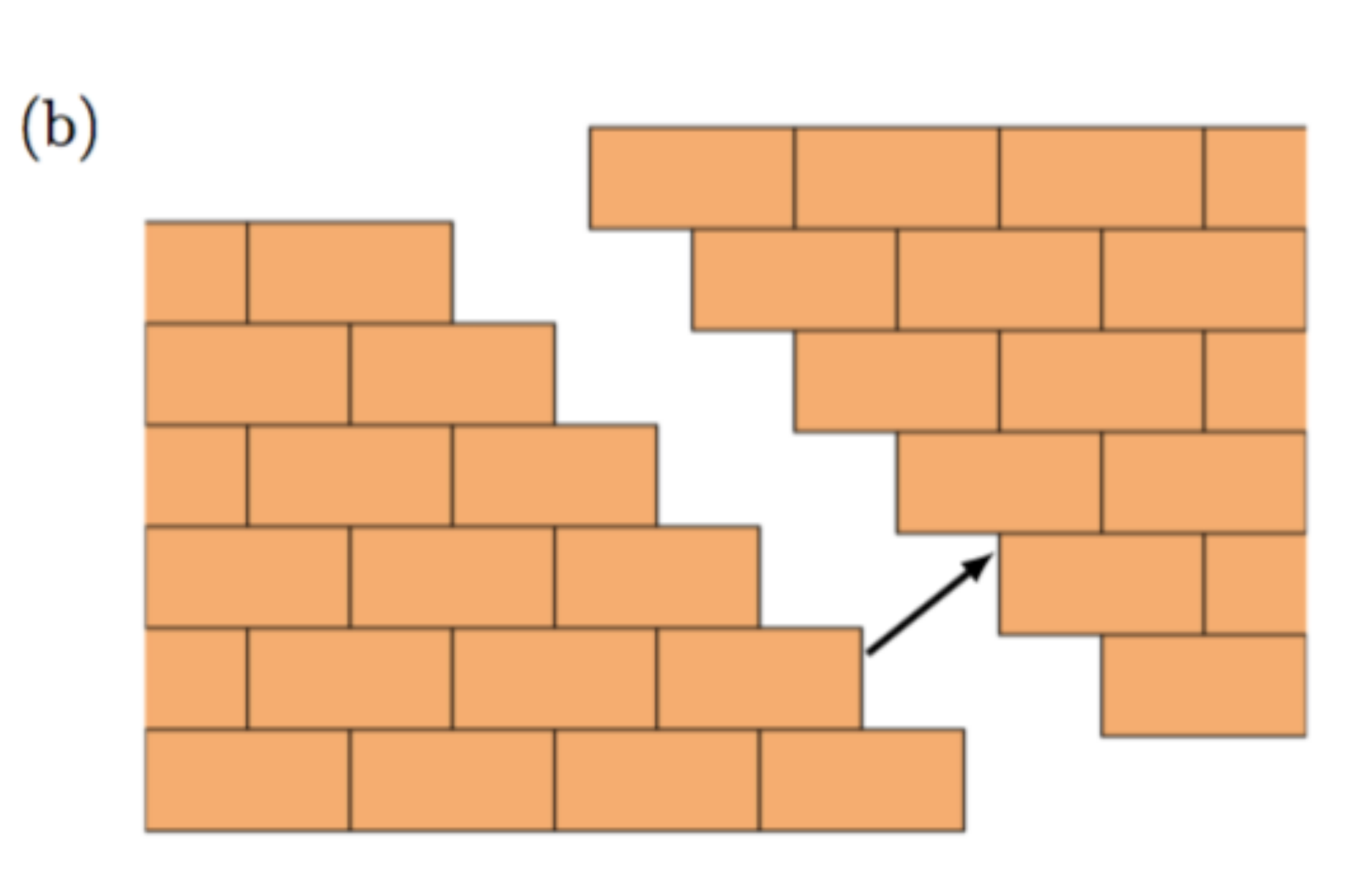}
	\caption{Setting of the problem: (a) reference configuration $\dmn$ is subdivided into a periodic collection of disconnected sets $\dmn \setminus \varepsilon\microgeom$ (blocks) by the periodic microstructure $\varepsilon\microgeom$ (interfaces) and (b) typical admissible displacement field, i.e.~in $\cU\subeps\at\dmn$, yielding fracture.}
\label{fig:geometry}
\end{figure}

Our aim is to study the asymptotic behavior of $\cfun\subeps$ as~$\varepsilon \rightarrow 0\p$ in the sense of~$\Gamma$-convergence. To this end, we define the (candidate) \emph{homogenized energy density} as given by the cell-problem
\begin{multline}
	\cden\hom\at{\mtx} = \inf\Big\{ 
		\frac{1}{2}\intvol{Y}{\norm{\cE u}^2} + \intsurf{J_u \cap Y}{\abs{u\p-u\m}} \bdot \\[0.5ex]
		u \in BD\txtsub{loc}\at{\R{\dmndim}},\, 
		J_{u} \subseteq \microgeom, \,
		\rbr{u\p-u\m} \odot \nu_u \in \cone\txtsub{0}  \,\, \cH^{n-1}\text{-a.e.}, \,
		u-u_{\mtx} \text{ }Y\text{-periodic}\Big\}
\label{eq:f_hom}
\end{multline}
for all $\mtx \in \symspace$. Moreover, we introduce the \emph{homogenized cone $\normalfont\cone\hom$ associated to the microgeometry $\microgeom$ and the cone~$\cone_{0}$} as the set
\begin{equation}
	\cone\hom = \dom\at{\cden\hom^{\infty}},
\label{eq:cone}
\end{equation}
i.e. the domain of the recession function~$\cden\hom^{\infty}$ of the homogenized energy density~$\cden\hom$. We denote by $\cone\hom^{\perp}$ the cone orthogonal to $\cone\hom$; i.e., 
\begin{equation}
	\cone\hom^{\perp} = \cbr{\mtxalt \in \symspace \bdot \, \inprod{\mtx}{\mtxalt} \leq 0 \text{ for all } \mtx \in \cone\hom},
\end{equation}
and by $P_{\cone\hom}$ [resp., $P_{\cone\hom^{\perp}}$] the orthogonal projection on $\cone\hom$ [resp., $\cone\hom^{\perp}$].

The homogenization theorem for functionals in~\eqref{eq:cohesive_functional}--\eqref{eq:test_displ} takes the following form.
\begin{theorem}
Let $\dmn$ be a bounded open set of $\R\dmndim$, let $\microgeom$ be a closed rectifiable $Y$-periodic $\rbr{\dmndim-1}$-dimensional subset of $\R\dmndim$. Moreover, let $\normalfont\cone\txtsub{0}$ be a closed cone in $\normalfont\symspace$ consisting of matrices of the form $a \odot b$ and satisfying condition~\eqref{eq:cone_0_hp}. Then the family $\rbr{F\subeps}$ of functionals defined by~\eqref{eq:cohesive_functional}--\eqref{eq:test_displ} $\Gamma$-converges on $BD\at\dmn \cap L^2\at{\dmn; \R\rngdim}$, with respect to the $L^2\at{\dmn; \R{\rngdim}}$-topology, to the functional $\normalfont \cfun\hom \bdot L^2\at{\dmn; \R{\rngdim}} \rightarrow \sbr{0, +\infty}$ given by
	\begin{equation}\normalfont
		\cfun\hom\at{u} = \begin{dcases}
			\intvol{\dmn} {\cden\hom\at{\cE u}} 
			+ \int_{\dmn} {\cden\hom^{\infty}\at{\diff{\wstn\sing u}{\abs{\wstn\sing u}}} \mathrm{d}\abs{\wstn\sing u}} & \text{ if } u \in \cU\hom\at{\dmn} \\
			+\infty \quad &\text{ otherwise},
		\end{dcases}
	\label{eq:hom_fun}
	\end{equation}
	where $\normalfont\cden\hom$ is the homogenized energy density in~\eqref{eq:f_hom} and the set $\normalfont \cU\hom\at\dmn$ of homogenized admissible displacements is 
	\begin{equation}\normalfont
		\cU\hom\at\dmn = \cbr{u \in BD\at{\dmn} \bdot P_{\cone\hom^{\perp}} \wstn\sing u = 0},
	\label{eq:hom_test_displ}
	\end{equation}
	with $\normalfont \cone\hom$ the homogenized cone associated to $\microgeom$ and $\cone_0$ defined by by~\eqref{eq:cone}.	
\label{thm:cohesive_main}
\end{theorem}

\section{A one-dimensional model problem}\label{s:1d_problem}
\begin{figure}
	\centering
	\includegraphics[clip = true, scale = 0.3]{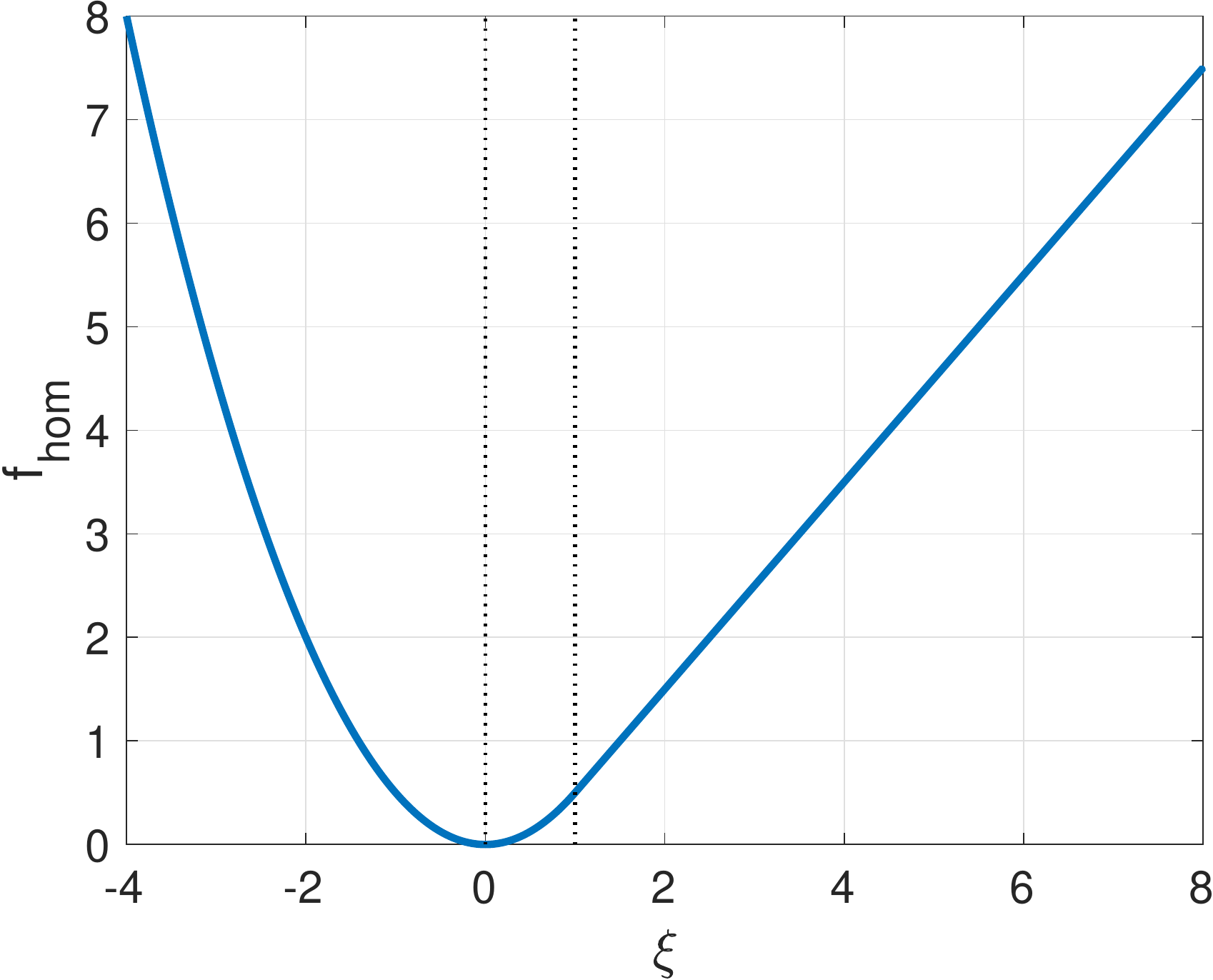}
	\caption{One-dimensional model problem: homogenized energy density.}
\label{fig:energy_density}
\end{figure}
\begin{figure}
	\centering
	\includegraphics[clip = true, scale = 0.6]{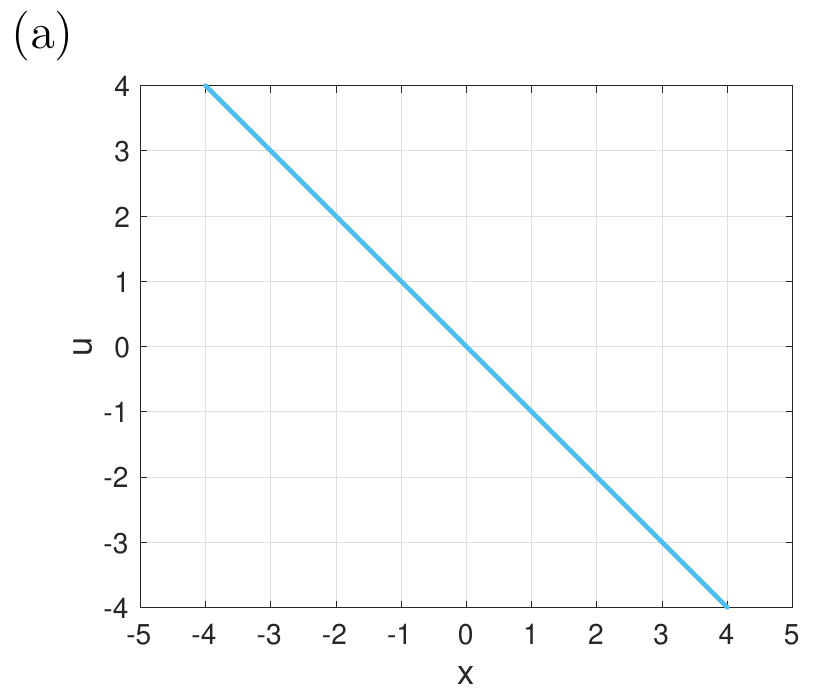}
	\hspace{1cm}
	\includegraphics[clip = true, scale = 0.6]{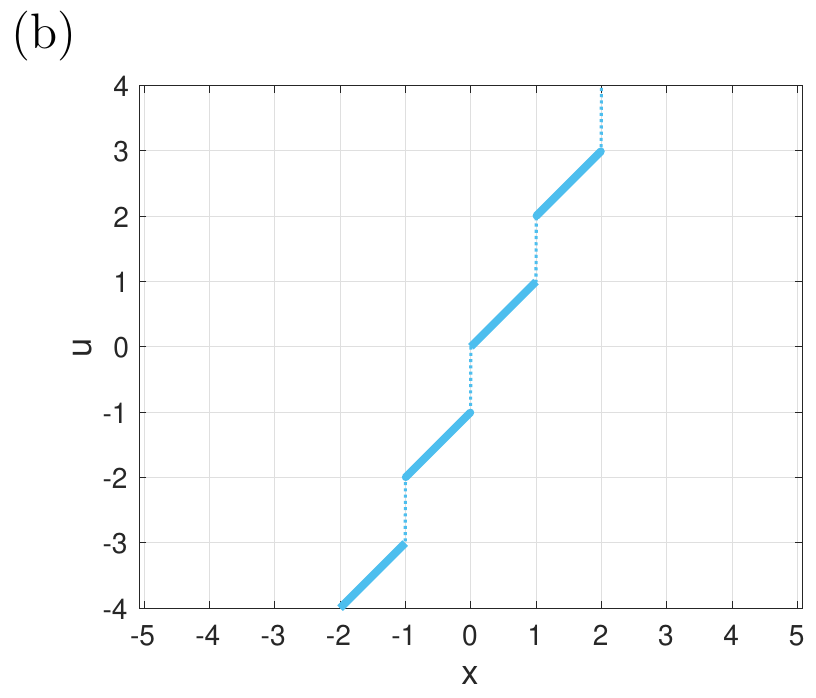}
	\caption{One-dimensional model problem: admissible displacement minimizer in the computation of the homogenized energy density. The strain is (a) $\mtx \leq 1$ and (b) $\mtx > 1$.}
\label{fig:minimizer}
\end{figure}

In this section we consider an illustration of the homogenization result stated in Theorem~\ref{thm:cohesive_main} referring to a one-dimensional setting. 
In such a context, the only choices for the microgeometry set~$\microgeom$ and the cone~$\cone_{0}$ involved in the unilateral constraint on the strain, are $\microgeom= \Z{}$ and $\cone_0 = [0, +\infty)$ respectively. Hence, the cell-problem~\eqref{eq:f_hom}, yielding the homogenized energy density $\cden\hom$, reduces to
\begin{multline}
	\cden\hom\at{\mtx} = \inf\Big\{ 
		\frac{1}{2}\intvol{Y\setminus\Z{}}{\rbr{u'}^2 } + \sum_{J_u \cap Y}\rbr{u\p - u\m} \bdot \\[-0.5ex]
		u \in H^1\at{\R{}\setminus\Z{}}, \,\,
		u\p-u\m \geq 0, \,\,
		u-u_{\mtx} \text{ }Y\text{-periodic}\Big\}
\label{eq:f_hom_1D}
\end{multline}
for all $\mtx \in \R{}$. Problem~\eqref{eq:f_hom_1D} turns out to be simplified when, instead of an admissible displacement function~$u$, its periodic part $\tilde{u} = u - u_{\mtx}$ is considered as unknown. Indeed, by an integration by parts, it follows that
\begin{multline}
	\cden\hom\at{\mtx} = \inf\Big\{ 
		\frac{1}{2}\intvol{Y\setminus \Z{}}{ \rbr{\tilde{u}'}^2} - \rbr{\mtx-1} \sum_{J_{\tilde{u}} \cap Y}\rbr{\tilde{u}\p - \tilde{u}\m} + \frac{1}{2}\,\mtx^2 \, \bdot  \\[-0.5ex]
		\tilde{u} \in H^1\at{\R{}\setminus\Z{}},\,\,
		\tilde{u}\p-\tilde{u}\m \geq 0, \,\,
		\tilde{u}\text{ }Y\text{-periodic}\Big\}.
\label{eq:1d_simpl}
\end{multline}

Assume first that $\mtx \leq 1$. Since the second term of the energy is positive for any admissible function~$\tilde{u}$, the minimum is attained in $H^1\at{\R{}}$. In particular, a minimizer is given by $\tilde u = 0$, and we obtain the solution 
\begin{equation}
	\cden\hom\at{\mtx} = \frac{1}{2}\,\mtx^2, \quad
	 u = u_{\mtx}, \quad
	\text{ for } \mtx \leq 1.
\end{equation}
Next we assume $\mtx > 1$. In such a case, because of their opposite sign, there is a competition between the first and the second energy terms. We claim that the function $\tilde{u}\at{x} = -\rbr{\mtx-1} x$ on $Y$, extended by $Y$-periodicity over $\R{}$, is a minimizer. In fact, for that choice the energy is equal to $\mtx-1/2$ and coincides with the estimate from below in the next Proposition~\ref{prop:growth}. Accordingly, we derive the solution
\begin{equation}
	\cden\hom\at{\mtx} = \mtx-\frac{1}{2}, \quad
	 u\at{x} = x + \mtx\ipbr{x}, \quad
	\text{ for } \mtx > 1,
\end{equation}
in which $\ipbr{x}$ denotes the integer part of $x$. The homogenized energy density $\cden\hom$ is depicted in Figure~\ref{fig:energy_density}, whereas Figure~\ref{fig:minimizer} shows the admissible displacement minimizer in its computation for strains (a) $\mtx \leq 1$ and (b) $\mtx > 1$. 

Upon observing that the recession function~$\cden\hom^{\infty}$ of the homogenized energy density~$\cden\hom$ is
\begin{equation}
	\cden\hom^{\infty}\at\mtx = \begin{dcases}
		+\infty & \text{ if } \mtx < 0 \\
		\mtx \quad &\text{ if } \mtx \geq 0,
	\end{dcases}
\end{equation}
from \eqref{eq:cone} the homogenized cone $\cone\hom$ associated to the microgeometry $\microgeom$ and the cone~$\cone_{0}$ results to be
\begin{equation}
	\cone\hom=[0, +\infty),
\end{equation}
in particular coinciding with $\cone_{0}$ itself. Hence, for a bounded open set~$\dmn$ of $\R{}$, the homogenized functional~$\cfun\hom$ given in \eqref{eq:hom_fun} is expressed as
\begin{equation}
	\cfun\hom\at{u} = \begin{dcases}
		\intvol{\dmn} {\cden\hom\at{u'}} 
		+ \wstn\sing u\at\dmn & \text{ if } u \in \cU\hom\at{\dmn} \\
		+\infty \quad &\text{ otherwise},
	\end{dcases}
\end{equation}
with the space $\normalfont \cU\hom\at\dmn$ of homogenized admissible displacements given in \eqref{eq:hom_test_displ} 
\begin{equation}
	\cU\hom\at\dmn = \cbr{u \in BV\at{\dmn; \R{}} \bdot \wstn\sing u \geq 0}.
\end{equation}

It is worth noticing some properties enjoyed by the homogenized energy density $\cden\hom$ in the present one-dimensional model problem:
\begin{enumerate}[label=(\roman*)]
	\item it exhibits a mixed growth. In fact, it is quadratic under compression $\mtx < 0$, where no cracking occurs, and linear under tension $\mtx > 1$, where cracking occurs accompanied by expense of cohesive fracture energy. Interestingly, in the tensile region $0 < \mtx < 1$ no fracture develops: in mechanical terms, that corresponds to the capability of the material to sustain moderate tensile stresses; in mathematical terms, the expense of fracture energy is not convenient compared to that of elastic energy under moderate tensile strain;
	\item it attains the upper bound value $\mtx^2/2$ over $\cone_0^\perp$; 
	\item it is of linear growth over $\cone\hom$. 
\end{enumerate}
Such properties have general validity, as will be shown in the next section.

\section{Some properties of the homogenized energy density}\label{s:properties}
This section is devoted to some properties of the homogenized energy density $\cden\hom$ defined in \eqref{eq:f_hom}. We first show that it satisfies a non-standard growth condition. That fact descends from the structure of the cell problem~\eqref{eq:f_hom}, as it consists in the sum of a quadratic elastic energy and a linear interface energy.
\begin{proposition}\label{prop:growth}
There exist constants $c_1, \,c_2 > 0$ such that the homogenized energy density $\normalfont \cden\hom$ satisfies
\begin{equation}\normalfont
	c_1 \rbr{\abs{\mtx} - 1} \leq \cden\hom\at\mtx \leq c_2 \abs{\mtx}^2,
\end{equation}
for all $\normalfont \mtx \in \symspace$.
\end{proposition}
\begin{proof}
Let $\mtx$ in $\symspace$ be fixed. The estimate from above is an immediate consequence of taking $u_{\mtx}$ as a test function in \eqref{eq:f_hom}:
\begin{equation}
	\cden\hom\at\mtx \leq \frac{1}{2} \norm{\mtx}^2 \leq \frac{M^2}{2}\abs{\mtx}^2.
\end{equation}
As for the estimate from below, we first enlarge the space of test functions by removing the kinematical constraints in  \eqref{eq:f_hom}:
\begin{equation}
	\cden\hom\at{\mtx} 
		\geq \inf\Big\{ \frac{1}{2}\intvol{Y}{\norm{\cE u}^2} + \intsurf{J_u \cap Y}{\abs{u\p-u\m}} \bdot 
		u \in BD\txtsub{loc}\at{\R{\dmndim}},\, 
		u-u_{\mtx} \text{ }Y\text{-periodic}\Big\}.
\label{eq:growth_0}
\end{equation}
Then, it suffices to observe that for a function $u$ in $BD\txtsub{loc}\at{\R{\dmndim}}$ such that $u-u_{\mtx}$ is $Y$-periodic, we have
\begin{align}
	\begin{aligned}
		\frac{1}{2}\intvol{Y}{\norm{\cE u}^2} + \intsurf{J_u \cap Y}{\abs{u\p-u\m}} &\\[1ex]
			&\hspace{-4cm}\geq c \rbr{\intvol{Y}{\abs{\cE u}} + \intsurf{J_u \cap Y}{\abs{u\p-u\m}}} - \frac{1}{2} \\[1ex]
			&\hspace{-4cm}\geq c \abs{\intvol{Y}{\cE u} + \intsurf{J_{u}\cap Y}{\rbr{u\p - u\m} \odot \nu_{u}}} - \frac{1}{2}
			= c \abs{\mtx} - \frac{1}{2},
	\end{aligned}
\label{eq:growth_1}
\end{align}	
with $c = \min\cbr{\sqrt{\alpha}, \, 1}$, where $\alpha$ is given by \eqref{eq:doppia_stima}. \qed
\end{proof}

Note that both estimates in Proposition~\ref{prop:growth} cannot be improved. Concerning the one from above, the following result holds.
\begin{proposition}
The homogenized energy density $\normalfont \cden\hom$  is such that
\begin{equation}\normalfont
	\cden\hom\at\mtx=\frac{1}{2}\norm{\mtx}^2, \quad
\end{equation}
for all $\mtx \in \cone_0^{\perp}$, where $\normalfont\cone_0^{\perp} = \cbr{\mtxalt \in \symspace \bdot \, \inprod{\mtx}{\mtxalt} \leq 0 \text{ for all } \mtx \in \cone_0}$ is the cone orthogonal to~$\cone_0$.
\end{proposition}
\begin{proof}
Let $\mtx \in \cone_0^\perp$ be fixed. For $u \in BD\txtsub{loc}\at{\R{\rngdim}}$ a test function in~\eqref{eq:f_hom}, we denote by $\tilde u = u - u_{\mtx}$ its $Y$-periodic part. We derive
\begin{align}
	\begin{aligned}
		\frac{1}{2}\intvol{Y}{\norm{\cE u}^2}  
			&= \frac{1}{2}\intvol{Y}{\norm{\cE \tilde u}^2} + \frac{1}{2} \norm{\mtx}^2 + \intvol{Y}{\inprodalt{\cE \tilde u}{\mtx}} \\[1ex]
			&\geq \frac{1}{2} \norm{\mtx}^2  
				- \intsurf{J_{u} \cap Y}{\inprodalt{\rbr{\tilde u\p - \tilde u\m}\odot \nu_{u}}{\mtx}}
			\geq \frac{1}{2} \norm{\mtx}^2,
	\end{aligned}
\end{align}
where the periodicity of $\tilde u$ has been used in the integration by parts and the fact that $\rbr{\tilde u\p - \tilde u\m}\odot \nu_{u} \in \cone_0 \, \cH^{\dmndim-1}\text{-a.e.}$, $\mtx \in \cone_0^\perp$.
\qed
\end{proof}
On the other hand, by definition~\eqref{eq:cone}, the energy density~$\cden\hom$ is sublinear over the homogenized cone~$\cone\hom$. We now supply an alternative characterization of such a cone. Heuristically, suppose that the minimum in~\eqref{eq:f_hom} is attained. Then, one might expect that over $\cone\hom$ the absolutely continuous part of the strain associated to the minimizer is almost everywhere vanishing. Hence, if there were no surface energy, over the cone~$\cone\hom$ we would have a vanishing energy density. Otherwise stated, $\cone\hom$ might be characterized as the kernel of the energy density obtained by dropping off the surface energy contribution instead as the cone where the homogenized energy density $\cden\hom$ is sublinear,. 

The argument above suggests to introduce an energy density~$\mden\hom$ by the following cell-problem formula
\begin{multline}
	\mden\hom\at{\mtx} = \inf\Big\{ 
		\frac{1}{2}\intvol{Y}{\norm{\cE u}^2} \bdot \\[0.5ex]
		u \in BD\txtsub{loc}\at{\R{\dmndim}},\, 
		J_{u} \subseteq \microgeom, \,
		\rbr{u\p-u\m} \odot \nu_u \in \cone\txtsub{0}  \,\, \cH^{n-1}\text{-a.e.}, \,
		u-u_{\mtx} \text{ }Y\text{-periodic}\Big\},
\label{eq:g_hom}
\end{multline}
and to define the associated kernel
\begin{equation}
	\mcone\hom = \cbr{\mtx \in \symspace \bdot \mden\hom\at{\mtx} = 0}.
\label{eq:masonry_cone}
\end{equation}
Note that the function~$\mden\hom$ and the cone~$\mcone\hom$ have been subject of investigation in~\cite{Braides_ChiadoPiat_2000}, in the context of a homogenization result for the purely degenerate case of dry-masonry structures. 

The major step for proving that in fact $\mcone\hom$ coincides with $\cone\hom$ is addressed in the following proposition, where it is shown that the homogenized energy density $\cden\hom$ is sublinear over~$\mcone\hom$. To this end, we introduce the technical assumption that the orthogonal cone $\cone^{\perp}$ of the convex hull $\cone$ of $\normalfont\cone\txtsub{0}$ has non-empty interior.

\begin{proposition}\label{thm:linear_growth} 
Let $\cone$ denote the convex hull of $\normalfont\cone\txtsub{0}$ and assume that the orthogonal cone $\cone^{\perp}$ has non-empty interior. Then there exists a constant $c > 0$ such that the homogenized energy density $\normalfont \cden\hom$ satisfies
\begin{equation}\normalfont
	\cden\txtsub{hom}\at{\mtx} \leq c \abs{\mtx}
\end{equation}
for all $\normalfont\mtx \in \mcone\hom$.
\end{proposition}
\begin{proof}
The proof will be carried out through several steps. \\[-2ex]

Step 1. \emph{If $\mtx$ is an interior point of $\cone$ such that $\norm{\mtx} = 1$, then there exists a constant $c_0 > 0$ such that $c_0 \norm{\mtxalt} \leq \inprod{\mtxalt}{-\mtx}$ for all $ \mtxalt \in \cone^{\perp}$.} 

Since $\mtx$ is an interior point of $\cone$, upon setting
	$\displaystyle c_0 = \min \cbr{\inprodalt{\mtxalt}{-\mtx} \bdot \,\, \mtxalt\in\cone^{\perp}, \,\,\norm{\mtxalt}=1 }$,
we have $c_0 > 0$ and the claim follows. \\[-2ex]

Step 2. \emph{There exists a constant $c > 0$ such that $\intsurf{Y \cap J_{ u}}{ \abs{{u}\p - {u}\m}} \leq c \intvol{Y}{\norm{\cE  u}}$ for all functions $\normalfont u \in BD\txtsub{loc}\at{\R{\dmndim}}$ which are $Y$-periodic and fulfill the unilateral condition $\rbr{{u}\p - {u}\m} \odot \nu_{u} \in \cone_{0} \,\,\cH^{n-1}$-a.e. on the jump set $J_{{u}}$.}

Let $\mtx$ be an interior point of  ${\cone}^{\perp}$ such that $\norm{\mtx}=1$. Integrating by parts and exploiting the $Y$-periodicity of $ u$, we have
\begin{align}
	\begin{aligned}
		\intvol{Y}{ \langle P_{\cone} \cE  u, - \mtx \rangle} 
			&= \intvol{Y}{ \langle \cE  u, - \mtx \rangle}
			- \intvol{Y}{ \langle P_{\cone^{\perp}} \cE  u, - \mtx \rangle} \\[1ex]
			&\leq -\intsurf{Y \cap J_{ u}}{ \langle \rbr{{u}\p - {u}\m} \odot \nu_{u}, - \mtx \rangle}
			+ \intvol{Y}{\Vert{P_{\cone^{\perp}} \cE  u}\Vert}.
	\end{aligned}
\end{align}
By Step 1, since $\cone^{\perp\perp} = \cone$, we derive 
\begin{equation}
	c_0 \intvol{Y}{\norm{ P_{\cone} \cE  u}} 
		+\intsurf{Y \cap J_{ u}}{ \langle \rbr{{u}\p - {u}\m} \odot \nu_{u}, - \mtx \rangle}
		\leq \intvol{Y}{\Vert{P_{\cone^{\perp}} \cE  u}\Vert}.
\label{eq:est_2}
\end{equation}
Analogously, as by assumption $\rbr{{u}\p - {u}\m} \odot \nu_{u} \in \cone_{0} \subseteq \cone$, we also have
\begin{equation}
	\langle \rbr{{u}\p - {u}\m} \odot \nu_{u}, - \mtx \rangle 
		\geq c_0 \norm{\rbr{{u}\p - {u}\m} \odot \nu_{u}}
		\geq c_1 \abs{{u}\p - {u}\m},
\label{eq:est_3}
\end{equation}
with $c_1 > 0$ a suitable constant. Hence, by \eqref{eq:est_2} and \eqref{eq:est_3} we obtain
\begin{equation}
	c_0 \intvol{Y}{\norm{ P_{\cone} \cE  u}}
		+ c_1 \intsurf{Y \cap J_{ u}}{ \abs{{u}\p - {u}\m}}
		\leq \intvol{Y}{\Vert{P_{\cone^{\perp}} \cE  u}\Vert}.
\end{equation}
By the continuity of the projection operator, we conclude 
\begin{equation}
	\intsurf{Y \cap J_{ u}}{ \abs{{u}\p - {u}\m}}
		\leq c \intvol{Y}{\Vert{P_{\cone^{\perp}} \cE  u}\Vert}
		\leq c \intvol{Y}{\norm{\cE  u} },
\end{equation}
with $c = 1/c_1$, which is the desired result. \\[-2ex]

Step 3. \emph{There exists a constant $c > 0$ such that $\normalfont \cden\txtsub{hom}\at{\mtx} \leq c \abs{\mtx}$ for all $\normalfont\mtx \in \mcone\hom$.}

Let $\mtx \in \mcone\hom$. By definition of the masonry homogenized energy density \eqref{eq:g_hom}, we can consider a minimizing sequence $\rbr{u_h}$ in $BD\txtsub{loc}\at{\R{\dmndim}}$ such that $J_{u_{h}} \subseteq \microgeom$, $\rbr{u_h\p-u_h\m} \odot \nu_{u_{h}} \in \cone\txtsub{0}\,\, \cH^{n-1}\text{-a.e.}$, $\tilde{u}_h = u_h-u_{\mtx}$ is $Y$-periodic and
\begin{equation}
	0 = g\hom\at{\mtx} \geq \frac{1}{2}\intvol{Y}{\norm{\cE u_h}^2} - \frac{1}{h^2}.
\end{equation}
For $t > 0$, we set $u_h^t = t u_h$. Since $u_h^t$ is a test function for the computation of the homogenized energy density $\cden\hom$, we get
\begin{align}
	\begin{aligned}
		\cden\hom\at{t\mtx} 
			&\leq \frac{1}{2}\intvol{Y}{\norm{\cE u_h^t}^2} + \intsurf{Y \cap J_{u_h^t}}{\vert{\rbr{u_h^t}\p-\rbr{u_h^t}\m}\vert\,} \\[0.5ex]
			&\leq \frac{t^2}{h^2} + t \intsurf{Y \cap J_{u_h}}{\abs{u_h\p-u_h\m}}.
	\end{aligned}
\label{eq:est_4}
\end{align}
Since $\tilde{u}_h$ is $Y$-periodic and $\rbr{\tilde{u}_{h}\p - \tilde{u}_{h}\m} \odot \nu_{{u}_{h}} \in \cone_{0} \,\,\cH^{n-1}$-a.e., by Step 2 we derive
\begin{align}
	\begin{aligned}
		\intsurf{Y \cap J_{u_h}}{\abs{u_h\p-u_h\m}}
			&\leq c \intvol{Y}{\norm{\cE \tilde{u}_h}}
			\leq c \intvol{Y}{\norm{\cE u_h}} + c \norm{\mtx}
			\leq \frac{c}{h} + c \norm{\mtx}.
	\end{aligned}
\end{align}
Finally, from \eqref{eq:est_4} we obtain
\begin{equation}
	\cden\hom\at{t\mtx} 
		\leq \liminf_h \rbr{\frac{t^2}{h^2} + c\,\frac{t}{h} + c\norm{t\mtx} }
		= c\norm{t\mtx}
\end{equation}
and by the arbitrariness of $t > 0$ the proof is concluded. \qed
\end{proof}

We are finally in position to conclude that $\mcone\hom$ coincides with $\cone\hom$.
\begin{corollary}
Let $\cone$ denote the convex hull of $\normalfont\cone\txtsub{0}$ and assume that the orthogonal cone $\cone^{\perp}$ has non-empty interior. Moreover, let $\normalfont\cone\hom$ and $\normalfont\mcone\hom$ be the two cones respectively defined in \eqref{eq:cone} and \eqref{eq:g_hom}--\eqref{eq:masonry_cone}. Then, $\normalfont \cone\hom = \mcone\hom$.
\end{corollary}
\begin{proof} 
By Proposition~\ref{thm:linear_growth}, we obtain that $\mcone\hom \subseteq \cone\hom$. For the opposite inclusion, we first observe that $\mden\hom$ is homogeneous of degree $2$ outside of $\mcone\hom$. Accordingly, since $\cden\hom \geq \mden\hom$, it follows that $\cden\hom$ has growth of order $2$ outside of $\mcone\hom$. Hence $\symspace \setminus \mcone\hom \subseteq \symspace \setminus \cone\hom$, and the proof is accomplished. 
\qed
\end{proof}

\section{Proof of the main theorem}\label{s:proof}
The proof of Theorem~\ref{thm:cohesive_main} will be obtained at the end of the section, as a consequence of the following propositions, which adapt to the present case the localization methods of $\prfGamma$convergence and homogenization. From now on, $\dmn$ will be a fixed bounded open subset of $\R\rngdim$.

\subsection{A compactness result}
In order to prove a compactness result for the integral functionals \eqref{eq:cohesive_functional}, we resort to the localization method of $\Gamma$-convergence \cite{DalMaso_1993, Braides_GC_2002}. Accordingly, we extend the definition of the functionals $\cfun\subeps$ explicitly highlighting the dependence on the open set of definition. Such functionals, defined on $L^2\at{\dmn;\R\rngdim} \times \cA\at\dmn$ and still denoted by $\cfun\subeps$, are then given by
\begin{equation}
	\cfun\subeps\at{u, \openset} = \begin{dcases}
		\frac{1}{2}\intvol{\openset \setminus \varepsilon \microgeom}{\norm{\cE u}^2}
		+ \intsurf{\openset \cap J_u}{\abs{u\p-u\m}}, & u \in \cU\subeps\at{\openset}, \\
		+\infty, \quad &\text{otherwise},
	\end{dcases}
\label{eq:cohesive_functional_loc}
\end{equation}
with the set of admissible displacements $\cU\subeps\at\openset$ given by \eqref{eq:test_displ} with $A$ in place of $\Omega$. The crucial result we prove is the following so-called {\em fundamental estimate} for the family $\rbr{\cfun\subeps}$.
\begin{proposition}\label{prop:fund_est}
For every $\eta > 0$ and for every $\openset', \openset'', \opensetalt \in \cA\at\dmn$ with $\openset' \subset\subset \openset''$, there exists a constant $M > 0$ with the following property: For every $\varepsilon > 0$ and for every $u \in L^2\at{A''; \R\rngdim}$, $v \in L^2\at{B; \R\rngdim}$ there exists a function $\varphi \in C_{0}^{\infty}\at{\openset''}$ with $\varphi = 1$ in a neighbourhood of $\bar{\openset}'$ and $0 \leq \varphi \leq 1$ such that
\begin{equation}
	F_\varepsilon\rbr{\varphi u + \rbr{1-\varphi}v, \openset' \cup \opensetalt} \leq
		\rbr{1+\eta} \sbr{ F_\varepsilon\rbr{u, \openset''} + F_\varepsilon\rbr{v, \opensetalt} }
		+ M \norm{u-v}_{L^2\rbr{S}}^2,
\end{equation}
where $S = \rbr{\openset'' \setminus \openset'} \cap \opensetalt$.
\end{proposition}
\begin{proof}
Let $\eta > 0$, $\openset'$, $\openset''$ and $\opensetalt$ be fixed as in the statement. Let $\openset_1, \dots{}, \openset_k$ be open sets satisfying the property $\openset' \subset\subset \openset_1 \subset\subset \openset_2 \subset\subset \dots{} \subset\subset \openset_{k+1} = \openset''$. For every $i \in \cbr{1, \dots{}, k}$, let $\varphi_i \in C^{\infty}_0\at{\openset_{i+1}}$ with $\varphi_i = 1$ on an open neighbourhood $V_i$ of $\overline\openset_{i}$ and $0 \leq \varphi_i \leq 1$. Let $\varepsilon > 0$ and consider $u \in  L^2\at{\openset''; \R{\rngdim}}$ and $v \in L^2\at{\opensetalt; \R{\rngdim}}$; in particular, we can assume that $u \in \cU\subeps\at{\openset''}$ and $v \in \cU\subeps\at{\opensetalt}$, and arbitrarily extend them respectively outside $\openset''$ and $\opensetalt$.
We set
\begin{equation} 
	w_i = \varphi_i u + \rbr{1-\varphi_i}v 
\end{equation}
on $\openset' \cup \opensetalt$ for every $i \in \cbr{1, \dots{}, k}$. We note that $w_i \in SBD\at{\openset' \cup \opensetalt} \cap L^2\at{\openset' \cup \opensetalt; \R{\rngdim}}$. Moreover, since $\cH^{\dmndim-1}$-a.e.~$J_{w_{i}} \subseteq J_{u} \cup J_{v} \subseteq \varepsilon \microgeom $, and
\begin{equation}
	\rbr{w_i\p - w_i\m} \odot \nu_{w_{i}} = 
		\varphi_i \rbr{u\p -u\m} \odot \nu_{u} + 
		\rbr{1-\varphi_i} \rbr{v\p - v\m} \odot \nu_{v}
\end{equation}
$\cH^{\dmndim-1}$-a.e.~on $J_{w_{i}}$, by the convexity assumption~\eqref{eq:cone_0_hp} on $\cone_0$,  we obtain that $\rbr{w_i\p - w_i\m} \odot \nu_{w_{i}} \in \cone_0$. Hence $w_i \in \cU\subeps\at{\openset' \cup \opensetalt}$ and we have
\begin{equation}
	\cfun\subeps\at{w_i, \openset' \cup \opensetalt}
		= \cfun\subeps\at{u, \rbr{\openset' \cup \opensetalt} \cap V_i} 
		+ \cfun\subeps\at{v, \opensetalt \setminus \spt \varphi_i} 
		+ \cfun\subeps\at{w_i, \opensetalt \cap \rbr{\openset_{i+1}\setminus\bar\openset_i}}.
\label{eq:fund_est_1}
\end{equation}
Let $T_i = \opensetalt \cap \rbr{\openset_{i+1}\setminus\bar\openset_i}$. We estimate the last term:
\begin{align}
	\begin{aligned}
		\cfun\subeps\at{w_i, T_i}
			&= \frac{1}{2}\intvol{T_i \setminus \varepsilon \microgeom}{\norm{\varphi_i \cE u + \rbr{1-\varphi_i}\cE v + \cE \varphi_i \odot \rbr{u-v}}^2} \\[1ex]
			&\hspace{0.5cm} + \intsurf{T_i \cap \rbr{J_{u}\setminus J_{v}}}{\abs{u\p-u\m}}
				+ \intsurf{T_i \cap \rbr{J_{v}\setminus J_{u}}}{\abs{v\p-v\m}} \\[1ex]
			&\hspace{0.5cm} + \intsurf{T_i \cap \rbr{J_{u} \cap J_{v}}}{\abs{\varphi_i \rbr{u\p -u\m} + \rbr{1-\varphi_i} \rbr{v\p - v\m}}}  \\[1ex]
			&\leq c \intvol{T_i \setminus \varepsilon \microgeom}{\rbr{\norm{\cE u}^2 + \norm{\cE v}^2 + \abs{\cE \varphi_i}^2\abs{u-v}^2}} \\[1ex]
			&\hspace{0.5cm} + \intsurf{T_i \cap J_{u}}{\abs{u\p-u\m}}
				+ \intsurf{T_i \cap J_{v}}{\abs{v\p-v\m}}\\[1ex]
			&\leq c\sbr{\cfun\subeps\at{u, T_i} + \cfun\subeps\at{v, T_i}} + c M' \norm{u-v}_{L^2\at{T_i; \R{\rngdim}}}^2,
	\end{aligned}
\end{align}
where we have set $
	M' = \max\limits_{1\leq i\leq k} \norm{\cE\varphi_i}_{L^2\at{T_i; \R{\rngdim}}}^2$.

As the sets $T_i$ are pairwise disjoint and $\bigcup_{i=1}^{k} T_i \subseteq \opensetalt \cap \rbr{\openset''\setminus\openset'} = S$,
there exists $i_0 \in \cbr{1, \dots{}, k}$ such that:
\begin{equation}
	\cfun\subeps\at{w_{i_{0}}, T_{i_{0}}} 
		\leq \frac{1}{k}\sum_{i=1}^{k} \cfun\subeps\at{w_i, T_i}
		\leq \frac{c}{k}\sbr{\cfun\subeps\at{u, \openset''} + \cfun\subeps\at{v, \opensetalt}} + M \norm{u-v}_{L^2\at{S; \R{\rngdim}}}^2,
\end{equation}
where $M = c M'/k$. From equation~\eqref{eq:fund_est_1}, it follows that:
\begin{equation}
	\cfun\subeps\at{w_{i_{0}}, \openset' \cup \opensetalt}
		= \rbr{1+\frac{c}{k}}\sbr{\cfun\subeps\at{u, \openset''} + \cfun\subeps\at{v, \opensetalt} }
		+ M \norm{u-v}_{L^2\at{S; \R{\rngdim}}}^2,
\end{equation}
and the proof is accomplished. \qed
\end{proof}

Next, we derive the following compactness result on the family $\rbr{\cfun\subeps}$.
\begin{proposition}\label{prop:compactness}
Let $\rbr{\varepsilon_h}$ be a sequence of positive numbers converging to $0$. Then there exists a subsequence $\rbr{\varepsilon_{\sigma\rbr{h}}}$ of $\rbr{\varepsilon_h}$ and a functional~$\cfun \bdot L^2\at{\dmn; \R\rngdim} \times \cA\at\dmn \rightarrow \sbr{0, +\infty}$ such that 
\begin{equation}
	\cfun\at{\cdot, \openset} = \prfGamma\lim_{h} \cfun_{\varepsilon_{\sigma\rbr{h}}}\at{\cdot, \openset}
\end{equation}
for every $\openset \in \cA_0\at\dmn$ with respect to the $L^2\at{\openset; \R{\rngdim}}$-topology. Moreover, for every $u \in L^2\at{\dmn; \R{\rngdim}}$, the set function $\cfun\at{u, \cdot}$ is the restriction to $\cA\at{\dmn}$ of a Borel measure on $\dmn$.
\end{proposition}
\begin{proof}
Using Proposition~\ref{prop:fund_est}, the proof follows from the general localization method of $\prfGamma$convergence (for an illustrative description of the method, see Chapter 16 in \cite{Braides_GC_2002}; a detailed proof of the method is given in~\cite{DalMaso_1993}, where, relying on the fundamental estimate, it is developed through Theorem 8.5, Theorem 14.23, Theorem 15.18 and Theorem 18.5). \qed

\end{proof}

\subsection{Integral representation on $H^1\at{\dmn; \R\rngdim}$}
On account of Proposition~\ref{prop:compactness}, we intend to identify the $\Gamma$-limit of a convergent sequence of functionals $\cfun\subeps$. Therefore, we assume that a sequence $\rbr{\varepsilon_h}$ of positive numbers converging to $0$ is given, such that for every $\openset \in \cA_0\at\dmn$ the limit
\begin{equation}
	\cfun\at{\cdot, \openset} = \prfGamma\lim_{h} \cfun\subepsh\at{\cdot, \openset}
\label{eq:Gamma_converging_subsequence}
\end{equation}
exists on $BD\at\openset \cap L^2\at{\openset; \R{\rngdim}}$. Unfortunately, in investigating a representation of the limit $\cfun$, we cannot directly resort to existing general results 
because the functionals $\cfun\subeps$ do not fulfill standard growth conditions on the whole space $BD\at\openset \cap L^2\at{\openset; \R{\rngdim}}$. To bypass this difficulty, we first restrict our attention to the behavior of $\cfun$ on $H^1\at{\openset; \R\rngdim}$, where the growth condition of order $2$ from above  can be exploited. Then, we extend such representation on $BD\at\openset \cap L^2\at{\openset; \R{\rngdim}}$ by convexity arguments.
 
A first result concerns the translation-invariance properties of the limit $\cfun$.
\begin{lemma}\label{lemma:translation_invariance}
Let $\cfun$ be defined as in \eqref{eq:Gamma_converging_subsequence}. Then, for every $\openset \in \cA_0\at\dmn$ and $u \in \dom\cfun\at{\cdot, \openset}$, the following properties hold:
\begin{equation}\normalfont
	\text{(i) } \cfun\at{u+a, \openset} = \cfun\at{u, \openset}, \quad\quad
	\text{(ii) } \cfun\at{\tau_{y}{u}, \tau_{y}{\openset}} = \cfun\at{u, \openset},
\end{equation}
for every $a \in \R\rngdim$ and $y \in \R\dmndim$, with $\rbr{\tau_{y}{u}}\at{x}=u\at{x-y}$ and $\tau_{y}{\openset} = \openset + y$.
\end{lemma}
\begin{proof}
These are general properties of the $\prfGamma$limit of periodic energies, which can be deduced with minor modifications e.g.~as in Lemma 3.7 in~\cite{Braides_Vitali_ARMA_1996}.   \qed
\end{proof}

By exploiting that the limit $\cfun$ satisfies a growth condition of order $2$ on $H^1\at{\dmn; \R{\rngdim}}$, we can then prove that it admits an integral representation. In particular, the relevant density function is convex, satisfies a growth condition of order $2$ and depends on the symmetric part of the gradient only.
\begin{proposition}\label{prop:Buttazzo_DalMaso}
There exists a unique convex function~$\cden \bdot \mtxspace \in [0, +\infty)$ enjoying the following properties:
\begin{enumerate}[label=\emph{(\roman*)}]
	\item $\cden\at{\mtx} \leq c \rbr{1 + \abs{\mtx}^2}$ for every $\mtx \in \mtxspace$, with $c > 0$ a suitable constant;
	\item $\cfun\at{u, \openset} = \intvol{\openset} {\cden\at{\cE u}}$ for every~$\openset \in \cA\at{\dmn}$ and $u \in H^1\at{\openset; \R{\rngdim}}$.
\end{enumerate}
\end{proposition}
\begin{proof} 
The functional $\cfun \bdot H^1\at{\dmn; \R{\rngdim}} \times \cA\at\dmn \rightarrow [0, +\infty)$ enjoys the assumptions required in Theorem 1.1 in \cite{Buttazzo_DalMaso_NATMA_1985}. Namely,~for every $u, \, v \in H^1\at{\dmn; \R{\rngdim}}$ and $\openset \in \cA\at\dmn$:
\begin{enumerate}[label={(\alph*)}]
	\item $\cfun\at{u, \openset} = \cfun\at{v, \openset}$ provided $\fixed{u}{\openset} = \fixed{v}{\openset}$;
	\item the set function $\cfun\at{u, \cdot}$ is the restriction to $\cA\at\dmn$ of a Borel measure on $\dmn$;
	\item $\cfun\at{u, \openset} = \cfun\at{u + a, \openset}$ for every $a \in \R{\rngdim}$;
	\item $\cfun\at{u, \openset} \leq c \intvol{\openset}{\rbr{1+\abs{\wnabla u}^2}}$, with $c$ a positive constant;
	\item $\cfun\at{\cdot,\openset}$ is sequentially weakly lower semicontinuous on $H^1\at{\dmn; \R{\rngdim}}$.
\end{enumerate}
Properties (b) and (c) follow from Proposition~\ref{prop:compactness} and Lemma~\ref{lemma:translation_invariance}, and properties (a), (d) and (e) are consequences of the representation of $\cfun\at{\cdot,\openset}$ as $\Gamma$-limit in~\eqref{eq:Gamma_converging_subsequence}. 

Hence, the Carath{\'e}odory function $\cden \bdot \R{\dmndim} \times \mtxspace \rightarrow [0, +\infty)$ defined by
\begin{equation}
	\cden\at{x, \mtx} = \limsup_{\rho \rightarrow 0} \frac{\cfun\at{u_{\mtx}, B_{\rho}\at{x}}}{\abs{B_{\rho}\at{x}}},
\label{eq:int_representation_1}
\end{equation}
gives the integral representation
\begin{equation}
	\cfun\at{u, \openset} = \intvol{\openset} {\cden\at{x, \wnabla u}},
\label{eq:int_representation_2}
\end{equation}
for every $\openset \in \cA\at{\dmn}$ and $u \in H^1\at{\openset; \R{\rngdim}}$. In particular, since $\cfun\at{\cdot,\openset}$ is convex for every $\openset \in \cA\at{\dmn}$, definition~\eqref{eq:int_representation_1} implies that the function $\cden$ is convex. Moreover, as a consequence of property (d),  the function $\cden$ satisfies the growth condition
\begin{equation}
	\cden\at{x,\mtx} \leq c\rbr{1 + \abs{\mtx}^2}.
\label{eq:growth_int_rep}
\end{equation}

We now show that $\cden$ is constant with respect to its first argument. Let $x$, $y \in \R\dmndim$ be fixed. Upon observing that $\tau_{y-x}u_{\mtx} = u_{\mtx} - \mtx \rbr{y-x}$ and $\tau_{y-x}B_{\rho}\at{x} = B_{\rho}\at{y}$, from Lemma~\ref{lemma:translation_invariance} we obtain
\begin{align}
	\begin{aligned}
		\cden\at{x, \mtx} 
			&= \limsup_{\rho \rightarrow 0} \frac{\cfun\at{u_{\mtx}, B_{\rho}\at{x}}}{\abs{B_{\rho}\at{x}}} 
			= \limsup_{\rho \rightarrow 0} \frac{\cfun\at{u_{\mtx} - \mtx \rbr{y-x}, B_{\rho}\at{y}}}{\abs{B_{\rho}\at{y}}}\\[1ex]
			&= \limsup_{\rho \rightarrow 0} \frac{\cfun\at{u_{\mtx}, B_{\rho}\at{y}}}{\abs{B_{\rho}\at{y}}}
			= \cden\at{y, \mtx}.
	\end{aligned}
\end{align}
Moreover, $\cden$ depends only on the symmetric part of the gradient, i.e.~$\cden\at{\mtx}=\cden\at{\mtxalt}$ whenever $\xi$, $\eta \in \mtxspace$ satisfy $\mtx\symsup = \mtxalt\symsup$. In fact, let $\rbr{u\subepsh}$ in $\cU\subepsh\at{\openset}$ be a sequence converging to $u_{\mtx}$ in $L^2\at{\openset; \R{\rngdim}}$ and such that  
\begin{equation}
	\lim_{h} \cfun\subepsh\at{u\subepsh, \openset}  
		= \cfun\at{u_{\mtx}, \openset} 
		= \intvol{\openset} {\cden\at{\mtx}} 
		= \abs{\openset} \cden\at{\mtx}.
\end{equation}			
Set $v\subepsh = u\subepsh + \rbr{\mtxalt-\mtx}x$ and note that~$v\subepsh \in \cU\subepsh\at{\openset}$, $v\subepsh$ converges to $u_{\mtxalt}$ in $L^2\at{\openset; \R{\rngdim}}$ and~$\cfun\subepsh\at{u\subepsh, \openset} = \cfun\subepsh\at{v\subepsh, \openset}$. Hence:
\begin{equation}
	\abs{\openset} \cden\at{\mtxalt}
		= \intvol{\openset} {\cden\at{\mtxalt}}
		= \cfun\at{u_{\mtxalt}, \openset} 
		\leq \liminf_{h}\cfun\subepsh\at{v\subepsh, \openset} 
		= \lim_{h}\cfun\subepsh\at{u\subepsh, \openset} 
		= \abs{\openset} \cden\at{\mtx},
\end{equation}
and by symmetry~$\cden\at{\mtxalt} = \cden\at{\mtx}$. Finally, (i) follows from~\eqref{eq:growth_int_rep} and (ii) follows from \eqref{eq:int_representation_2}, whereas the uniqueness of $\cden$ follows from~\eqref{eq:int_representation_1}. \qed
\end{proof}

\subsection{Characterization of the homogenized energy density} 
In the previous section we have proven that the $\prfGamma$limit $\cfun$ of a convergent sequence of functionals $\rbr{\cfun\subepsh}$ admits an integral representation on $H^1\at{\dmn; \R{\rngdim}}$. We are now in a position to show that the density $\cden$ indeed coincides with the homogenized energy density $\cden\hom$ defined in \eqref{eq:f_hom}. In particular, that implies that $\cden$ does not depend on the sequence $\rbr{\varepsilon_h}$.

\begin{proposition}\label{thm:f_hom_leq}
	$\normalfont \cden\at{\mtx} \leq \cden\hom\at{\mtx}$ for every $\mtx \in \mtxspace$.
\end{proposition}
\begin{proof}
For a fixed $\mtx \in \mtxspace$, let $u \in BD\txtsub{loc}\at{\R{\dmndim}}$ be a test function for the computation of $\cden\hom$, i.e.~such that $J_u \subseteq \microgeom$, $\rbr{u\p-u\m}\odot \nu_u \in \cone_0 \,\, \cH^{n-1}\text{-a.e.}$ and $\tilde{u} = u-u_{\mtx}$ is $Y$-periodic. Moreover, let $\rbr{\varepsilon_h}$ be a sequence of positive numbers converging to $0$. Upon noticing that the sequence of scaled functions~$u\subepsh\at{x} = \varepsilon_h u\at{x/\varepsilon_h}$ converges to $u_{\mtx}$ in $L^2\at{Y; \R{\rngdim}}$, by the liminf inequality of the $\Gamma$-convergence, it follows that
\begin{align}
	\begin{aligned}
		\cden\at{\mtx} 
			&= \cfun\at{u_{\mtx}, Y} \leq \liminf_{h} \cfun\subepsh\at{u\subepsh, Y}
			= \liminf_{h} \rbr{ 
				\frac{1}{2}\intvol{Y} {\norm{{ \cE u\subepsh}}^2}
				+\intsurf{Y} {\abs{u\subepsh\p - u\subepsh\m}}}.
	\end{aligned}
\label{eq:upper_estimate}
\end{align}
For the bulk energy term we obtain
\begin{align}
	\begin{aligned}
		\frac{1}{2}\intvol{Y} {\norm{{ \cE u\subepsh}}^2}
			= \frac{\varepsilon_h^n}{2} \intvol{Y/\varepsilon_h} { \norm{{\cE u}}^2}  
			\leq \frac{\varepsilon_h^n}{2} \ipbr{1+\frac{1}{\varepsilon_h}}^n \intvol{Y} { \norm{{\cE u}}^2},
	\end{aligned}
\end{align}
whereas for the surface energy term
\begin{align}
	\begin{aligned}
		\intsurf{Y} {\abs{u\subepsh\p - u\subepsh\m}}
			= \varepsilon_h^{n} \int_{Y/\varepsilon_h} {\abs{u\p - u\m} \mathrm{d}\cH^{n-1}}
			\leq \varepsilon_h^n \ipbr{1+\frac{1}{\varepsilon_h}}^n \intsurf{Y} {\abs{u\p - u\m}}.
	\end{aligned}
\end{align}
Finally, from \eqref{eq:upper_estimate} we get
\begin{equation}
	\cden\at{\mtx}  \leq \frac{1}{2}\intvol{Y} { \norm{{\cE u}}^2} + \intsurf{Y} {\abs{u\p - u\m}},
\end{equation}
whence, taking the supremum over functions~$u \in BD\txtsub{loc}\at{\R{\dmndim}}$ such that $J_u \subseteq \microgeom$, $\rbr{u\p-u\m}\odot \nu_u \in \cone_0 \,\, \cH^{n-1}\text{-a.e.}$, we obtain the desired result. \qed
\end{proof}

\begin{proposition}\label{thm:f_hom_geq}
	$\normalfont \cden\hom\at{\mtx} \leq \cden\at{\mtx}$ for every $\mtx \in \mtxspace$.
\end{proposition}
\begin{proof}
Let $\mtx \in \mtxspace$. We can consider a sequence~$\rbr{u\subepsh}$, with $u\subepsh \in \cU\subepsh\at{Y}$ for all $h \in \N$, converging to $u_{\mtx}$ and such that
\begin{equation}
	\cden\at{\mtx}
		= \intvol{Y}{\cden\at{\mtx}}
		= \cfun\at{u_{\mtx}, Y}
		= \lim_{h} \cfun\subepsh\at{u\subepsh, Y}.
\end{equation}
Let $\delta > 0$. We define~$\phi\at{y} = \max\cbr{\rbr{\delta^{-1}\dist\at{y,\partial Y}}, 1}$ and~$S_{\delta} = \cbr{y \in Y \bdot \dist\at{y,\partial Y} < \delta}$. Moreover we set
\begin{equation}
	v\subepsh = \phi \, u\subepsh + \rbr{1-\phi}u_{\mtx}, 
\end{equation}
and consider the extension of~$v\subepsh$ to a function defined on~$\varepsilon_h \ipbr{1 + 1/\varepsilon_h} Y$ as
\begin{equation}
	w\subepsh\at{y} = \begin{dcases}
		v\subepsh\at{y}, & \text{ if } y \in Y, \\
		u_{\mtx}\at{y}, & \text{ if } y \in \varepsilon_h \ipbr{1 + 1/\varepsilon_h}Y \setminus Y.
	\end{dcases}
\end{equation}
The function~$w\subepsh$ is extended to all~$\R{\dmndim}$ by requiring~$\tilde{w}\subepsh =w\subepsh - u_{\mtx}$ to be $\varepsilon_h \ipbr{1 + 1/\varepsilon_h} Y$-periodic. By the continuity of~$u_{\mtx}$, it follows that~$w\subepsh \in \cU\subepsh\at{Y}$.  We now claim that the function
\begin{equation}
	z\subepsh\at{y} = \frac{1}{\varepsilon_h \ipbr{1 + 1/\varepsilon_h}^n} \sum_{k \in \cbr{0, \dots, \ipbr{1/\varepsilon_h}}^n} w\subepsh\at{\varepsilon_h y + \varepsilon_h k},
\end{equation}
is an admissible test function for the computation of~$\cden\hom\at{\mtx}$. In fact, by the Y-periodicity of $\microgeom$ and since~$J_{w\subepsh} \subseteq \varepsilon_h \microgeom$, we derive
\begin{equation}
	J_{z\subepsh}  
		\subseteq \bigcup_{k \in \cbr{0, \dots, \ipbr{1/\varepsilon_h}}^n} J_{w\subepsh\rbr{\varepsilon_h \,\cdot\, + \varepsilon_h k}}
		\subseteq \bigcup_{k \in \cbr{0, \dots, \ipbr{1/\varepsilon_h}}^n} \rbr{-k + J_{w\subepsh}/\varepsilon_h}
		\subseteq \microgeom;
\end{equation}
moreover, since $\rbr{w\subepsh\p - w\subepsh\m} \odot \nu_{w\subepsh} \in \cone_{0} \,\, \cH^{n-1}\text{-a.e.}$, it follows that $\rbr{z\subepsh\p - z\subepsh\m} \odot \nu_{z\subepsh} \in \cone_{0} \,\, \cH^{n-1}\text{-a.e}$. In order to show that $\tilde{z}\subepsh = z\subepsh - u_{\mtx}$ is Y-periodic, it suffices to notice that, for $r \in \cbr{0, \dots, \ipbr{1/\varepsilon_h}}$ and $e_l$ denoting the $l$-th vector of the canonical base of $\R{\dmndim}$, we have
\begin{align}
	\begin{aligned}
		\tilde{z}\subepsh\at{y+re_j} 
			&=  \frac{1}{\varepsilon_h \ipbr{1 + 1/\varepsilon_h}^n} \sum_{k \in \cbr{0, \dots, \ipbr{1/\varepsilon_h}}^n} \cbr{ \tilde{w}\subepsh\at{\varepsilon_h y + \varepsilon_h \rbr{k+re_j}} + \varepsilon_h  \mtx k } \\
			&=  \frac{1}{\varepsilon_h \ipbr{1 + 1/\varepsilon_h}^n} \sum_{k \in \cbr{0, \dots, \ipbr{1/\varepsilon_h}}^n} \cbr{ \tilde{w}\subepsh\at{\varepsilon_h y + \varepsilon_h k} + \varepsilon_h  \mtx k }
			= \tilde{z}\subepsh\at{y},
	\end{aligned}
\end{align}
where the $\varepsilon_h \ipbr{1 + 1/\varepsilon_h} Y$-periodicity of $\tilde{w}\subepsh$ has been exploited. 

Using $z\subepsh$ as a test function in the computation of $\cden\hom\at\mtx$, it follows that
\begin{equation}
	\cden\hom\at{\mtx} \leq \frac{1}{2}\int_{Y} { \norm{\cE z\subepsh}^2 \mathrm{d}y} 
		+ \intsurf{Y \cap J_{z\subepsh}}{\abs{z\subepsh\p-z\subepsh\m}},
\label{eq:lower_estimate_1}
\end{equation}
and we have to estimate the two terms. 
For the bulk energy term, by the $Y$-periodicity of~$\tilde{z}\subepsh$ and by the $\varepsilon_h \ipbr{1 + 1/\varepsilon_h} Y$-periodicity of $\tilde{w}\subepsh$ we derive 
\begin{align}
	\begin{aligned}
			\int_{Y} { \norm{\cE z\subepsh}^2 \mathrm{d}y} 
			&= \frac{1}{\ipbr{1+{1/\varepsilon_h}}^n} \int_{\ipbr{1 + \frac{1}{\varepsilon_h}} Y} { \norm{\cE z\subepsh}^2 \mathrm{d}y} \\[1ex]
			&\leq \frac{1}{\ipbr{1+1/\varepsilon_h}^{n}} \int_{\varepsilon_h\ipbr{1 + \frac{1}{\varepsilon_h}} Y} {\norm{{ \cE w\subepsh}}^2 \mathrm{d}y} \\[1ex]				
			&\leq \int_{Y} {\norm{{ \cE v\subepsh}}^2 \mathrm{d}y} + \norm{\mtx\symsup}^2 \rbr{\varepsilon_h^n\ipbr{1 + {1\over\varepsilon_h}}^n - 1}.
	\end{aligned}
\end{align}
In particular, for~$\eta > 0$ and $c_\eta > 0$ a suitable constant, the first term can be estimated by
\begin{align}
	\begin{aligned}
		\int_{Y} {\norm{{ \cE v\subepsh}}^2 \mathrm{d}y} 
			&= \int_{Y \setminus S_{\delta}} {\norm{{ \cE u\subepsh}}^2 \mathrm{d}y}
				+\int_{S_{\delta}} {\norm{{ \phi \cE u\subepsh + \rbr{1-\phi}\mtx\symsup +\nabla \phi \odot \rbr{u\subepsh - u_{\mtx}} }}^2 \mathrm{d}y} \\[1ex]	
			&\leq \rbr{1+\eta}\int_{Y} {\norm{{ \cE u\subepsh}}^2 \mathrm{d}y}	
				+c_\eta \abs{S_{\delta}} \norm{\mtx\symsup}^2 
				+ \frac{c_\eta}{\delta^2} \int_{S_{\delta}} { \abs{u\subepsh - u_{\mtx}}^2 \mathrm{d}y}.
	\end{aligned}
\end{align}
Analogously, for the surface energy term the $Y$-periodicity of~$\tilde{z}\subepsh$ implies
\begin{align}
	\begin{aligned}
			\intsurf{Y \cap J_{z\subepsh}}{\abs{z\subepsh\p-z\subepsh\m}}
			&= \frac{1}{\ipbr{1+1/\varepsilon_h}^n} \intsurf{\ipbr{1 + \frac{1}{\varepsilon_h}} Y} { \abs{z\subepsh\p - z\subepsh\m}}\\[1ex]
			&= \frac{1}{\ipbr{1+1/\varepsilon_h}^{n}} \intsurf{\varepsilon_h\ipbr{1 + \frac{1}{\varepsilon_h}} Y} {\abs{w\subepsh\p - w\subepsh\m}} \\[1ex]
			&\leq \intsurf{Y} {\abs{u\subepsh\p - u\subepsh\m}}.
	\end{aligned}
\end{align}
Accordingly, from \eqref{eq:lower_estimate_1} we obtain
\begin{align}
	\begin{aligned}
		\cden\hom\at{\mtx} 
			&\leq \frac{1}{2}\rbr{1+\eta}\int_{Y} {\norm{{ \cE u\subepsh}}^2 \mathrm{d}y}	
			+\frac{c_\eta}{2} \abs{S_{\delta}} \norm{\mtx\symsup}^2 
			+ \frac{c_\eta}{2\delta^2} \int_{S_{\delta}} { \abs{u\subepsh - u_{\mtx}}^2 \mathrm{d}y} \\[1ex]
			& \quad+ \norm{\mtx\symsup}^2 \rbr{\varepsilon_h^n\ipbr{1 + {1\over\varepsilon_h}}^n - 1}
			+\intsurf{Y} {\abs{u\subepsh\p - u\subepsh\m}}.
	\end{aligned}
\end{align}
As $u\subepsh$ converges to $u_{\mtx}$ in $L^2\at{Y; \R\rngdim}$ and~$\abs{S_{\delta}} = 1-\rbr{1-2\delta}^2 \rightarrow 0$ for~$\delta \rightarrow 0\p$, by letting first $h \rightarrow + \infty$, then~$\delta \rightarrow 0\p$ and finally~$\eta \rightarrow 0\p$, we get
\begin{equation}
	\cden\hom\at{\mtx} 
		\leq \lim_{h} 
		\rbr{\frac{1}{2}\int_{Y} {\norm{{ \cE u\subepsh}}^2 \mathrm{d}y} +\int_{Y} {\abs{u\subepsh\p - u\subepsh\m}}}
		= \lim_{h} \cfun\subepsh\at{u_{\varepsilon_h}, Y} 
		= \cden\at{\mtx},
\end{equation}
and the proof is concluded. \qed
\end{proof}

We conclude this section with a corollary immediately descending from Propositions \ref{prop:Buttazzo_DalMaso}, \ref{thm:f_hom_leq} and \ref{thm:f_hom_geq}, which characterizes the $\prfGamma$limit $\cfun$ of a convergent sequence of functionals $\rbr{\cfun\subepsh}$ on Sobolev functions.
\begin{corollary}[($\prfGamma$convergence on $H^1$)]
For every~$\openset \in \cA_{0}\at{\dmn}$, the family of functionals $\rbr{\cfun\subepsh\at{\cdot, \openset}}$ $\prfGamma$converges on $H^1\at{\openset; \R{\rngdim}}$, with respect to the $L^2\at{\dmn;\R\rngdim}$-topology, to the functional $\normalfont \cfun\hom\at{\cdot, \openset}$.
\label{cor:int_rep}
\end{corollary}

\subsection{Characterization of the homogenized functional}
In this section we conclude the proof of Theorem~\ref{thm:cohesive_main} by showing that $\cfun\hom\at{\cdot, \dmn}$ is the $\prfGamma$limit of $\cfun\subeps\at{\cdot, \dmn}$ on $BD\at\dmn \cap L^2\at{\dmn; \R\rngdim}$.

As a preliminary step, we prove a lower-semicontinuity result on the functional $\cfun\hom$.
\begin{proposition}\label{thm:F_hom_lsc}
	For every $\openset \in \cA_0\at\dmn$, the functional $\normalfont \cfun\hom\at{\cdot, \openset}$ is lower semicontinuous on $BD\at\openset \cap L^2\at{\openset; \R\rngdim}$ with respect to the weak convergence in $BD\at\openset$.
\end{proposition}
\begin{proof}
Let $\openset \in \cA_0\at\dmn$. We consider a sequence $\rbr{u_h}$ in $BD\at\openset \cap L^2\at{\openset; \R\rngdim}$ weakly converging to a function $u$ in $BD\at\openset$. Accordingly, there exists a subsequence~$\rbr{u_{h_{j}}}$ such that $\wstn u_{h_{j}}$ weakly$^{*}$ converges to $\wstn u$ in $\cM\at{\openset; \mtxspace}$.
For $\lambda > 0$, let 
\begin{equation}
	\varphi_{\lambda}\at{\mtx} = \inf_{\mtxalt \in \mtxspace} \cbr{\cden\hom\at{\mtxalt} +  \lambda \abs{\mtx-\mtxalt}}
\end{equation}
denote the infimal convolution of $\cden\hom$ and $\lambda\abs{\cdot}$. Then $\varphi_{\lambda}$ enjoys the following properties: (i) $\varphi_{\lambda} \leq \lambda\abs{\cdot}$, (ii) $\varphi_{\lambda}$ is convex and (iii) $\varphi_{\lambda}$ increasingly converges to $\cden\hom$ for $\lambda \uparrow \infty$. Using a classical theorem by Reshetnyak on convex functional of measures \cite{Reshetnyak_SMJ_1968}, properties (i) and (ii) imply that
\begin{equation}
	\int_{\openset} {\varphi_{\lambda}\at{\wstn u}} 
		\leq \liminf_h \int_{\openset} {\varphi_{\lambda}\at{\wstn u_h}}
		\leq \liminf_h \cfun\hom\at{u_h, \openset},
\end{equation}
whence by Fatou's lemma
\begin{equation}
	\cfun\hom\at{u, \openset} \leq \liminf_h \cfun\hom\at{u_h, \openset},
\end{equation}
which concludes the proof. \qed
\end{proof}

We are now in position to prove the $\prfGamma$liminf inequality. A convexity method through convolution is employed to exploit the integral representation of the $\prfGamma$limit on Sobolev functions.
\begin{proposition}\label{thm:liminf}
Let $\rbr{\varepsilon_h}$ be a sequence of positive numbers converging to $0$. Then there exists a subsequence $\rbr{\varepsilon_{\sigma\at{h}}}$ of $\rbr{\varepsilon_h}$ such that
$\normalfont\cfun\hom\at{u, \dmn} \leq \prfGamma\liminf_{h} \cfun_{\varepsilon_{\sigma\at{h}}}\at{u, \dmn}$ for every $\normalfont u \in BD\at\dmn \cap L^2\at{\dmn; \R\rngdim}$.
\end{proposition}
\begin{proof}
Let $\rbr{\openset_j}$ be a sequence in $\cA_{0}\at\dmn$, with $\openset_j \subset\subset \dmn$ for all $j \in \N$, converging increasingly to~$\dmn$. By Proposition~\ref{prop:compactness}, there exists a subsequence $\rbr{\varepsilon_{\sigma\rbr{h}}}$ of $\rbr{\varepsilon_h}$ such that
\begin{equation}
	\cfun\at{\cdot,\openset} = \prfGamma\lim_h \cfun_{\varepsilon_{\sigma\rbr{h}}}\at{\cdot,\openset}
\end{equation}
for all $\openset = \openset_j$ with $j \in \N$, and for $\openset = \dmn$. 
Now, fix $u \in BD\at\dmn \cap L^2\at{\dmn; \R\rngdim}$ and let $\rbr{\rho_h}$ be a sequence of mollifiers such that $\rho_h$ has support in the open ball in $\R\dmndim$ of center $0$ and radius $1/h$, denoted $B_h$. Setting $u_h = \rho_h \ast u $, we notice that $u_h \in C^{\infty}_{0}\at{\dmn; \R\rngdim}$, $u_h$ converges to $u$ in $L^2\at{\dmn; \R\rngdim}$ and $\wstn u_h$ weakly$^*$ converges to $\wstn u$ in $\cM\at{\dmn; \mtxspace}$, i.e.~$u_h$ weakly converges to $u$ in $BD\at\dmn$. By Corollary \ref{cor:int_rep}, we derive
\begin{equation}
	\cfun\hom\at{u_h, \openset_j}
	= \intvol{\openset_j}{\cden\hom\at{\cE u_h}}
	= \cfun\at{u_h, \openset_j}.
\label{eq:liminf_1}	
\end{equation}
Next, we observe that the functional $\cfun\at{\cdot, \openset}$ is lower semicontinuous and convex on $L^2\at{\dmn; \R\rngdim}$ for every $\openset \in \cA_{0}\at\dmn$. For $h \in \N$ such that $1/h < \dist\at{A_j, \partial\dmn}$, Jensen's inequality implies that
\begin{equation}
	\cfun\at{u_h, \openset_j}
		= \cfun\at{\int_{B_h}{ \rho_h\at{y} \tau_{y}{u} \,\mathrm{d}y},  \openset_j}
		\leq \int_{B_h}{ \rho_h\at{y} \cfun\at{\tau_{y}{u},  \openset_j} \mathrm{d}y}.
\end{equation}
Since $\cfun$ is translation invariant (Lemma \ref{lemma:translation_invariance}), we obtain
\begin{equation}
	\cfun\at{u_h, \openset_j} \leq \int_{B_h}{ \rho_h\at{y} \cfun\at{u, \tau_{-y}\openset_j} \mathrm{d}y}.
\end{equation}
In particular, because $\tau_{-y}\openset_j \subseteq \dmn$ for every $y \in B_h$, we get
\begin{equation}
	\cfun\at{u_h, \openset_j} \leq \int_{B_h}{ \rho_h\at{y} \cfun\at{u, \dmn} \mathrm{d}y} = \cfun\at{u, \dmn}.
\label{eq:liminf_2}
\end{equation}
By Proposition~\ref{thm:F_hom_lsc}, the functional $\cfun\hom$ is lower semicontinuous on $BD\at{\openset_j} \cap L^2\at{\openset_j; \R\rngdim}$ with respect to the weak convergence in $BD\at{\openset_j}$. Then, from \eqref{eq:liminf_1} and \eqref{eq:liminf_2} we derive
\begin{equation}
	\cfun\hom\at{u, \openset_j} \leq \liminf_h \cfun\hom\at{u_h, \openset_j} \leq \cfun\at{u, \dmn}
\label{eq:dom}
\end{equation}
for all $j \in \N$. Taking the supremum with respect to $j \in \N$, we conclude
\begin{equation}
	\cfun\hom\at{u, \dmn} \leq \cfun\at{u, \dmn} = \prfGamma\liminf_h \cfun_{\varepsilon_{\sigma\at{h}}}\at{u,\dmn},
\end{equation}
and the proof is accomplished. \qed
\end{proof}

Next we prove the $\prfGamma$limsup inequality.
\begin{proposition}\label{thm:limsup}
Let $\rbr{\varepsilon_h}$ be a sequence of positive numbers converging to $0$. Then there exists a subsequence $\rbr{\varepsilon_{\sigma\at{h}}}$ of $\rbr{\varepsilon_h}$ such that
$\prfGamma\limsup_{h} \cfun_{\varepsilon_{\sigma\at{h}}}\at{u, \dmn} \leq \normalfont\cfun\hom\at{u, \dmn}$ for every $\normalfont u \in BD\at\dmn \cap L^2\at{\dmn; \R\rngdim}$.
\end{proposition}
\begin{proof}
By Proposition~\ref{prop:compactness}, there exists a subsequence $\rbr{\varepsilon_{\sigma\rbr{h}}}$ of $\rbr{\varepsilon_h}$ such that
\begin{equation}
	\cfun\at{\cdot,\openset} = \prfGamma\lim_h \cfun_{\varepsilon_{\sigma\at{h}}}\at{\cdot,\openset}
\end{equation}
for all $\openset \in \cA_0\at\dmn$ on $L^2\at{\openset; \R\rngdim}$. If $u \in L^2\at{\dmn; \R\rngdim} \setminus \cU\hom\at\dmn$, the claim is accomplished because $\cfun\hom\at{u, \dmn} = + \infty$; then, fix $u \in \cU\hom\at{\dmn}$. As in the previous proof, let $\rbr{\rho_j}$ be a sequence of mollifiers such that $\rho_j$ has support in $B_j$ and set $u_j = \rho_j \ast u$. Then $u_j \in C^{\infty}_{0}\at{\dmn; \R\rngdim}$ and $u_j$ converges to $u$ in $L^2\at{\dmn; \R\rngdim}$. By the lower semicontinuity of the~$\Gamma$-limsup and Corollary \ref{cor:int_rep}, we derive
\begin{equation}
	\prfGamma\limsup_{h} \cfun_{\varepsilon_{\sigma\at{h}}}\at{u,\dmn}
		\leq \liminf_j \rbr{ \prfGamma\limsup_{h} \cfun_{\varepsilon_{\sigma\at{h}}}\at{u_j,\dmn} } 
		= \liminf_j \cfun\hom\at{u_j, \dmn}.
\end{equation}
Then, Lemma 5.2 in \cite{Temam_1985} implies
\begin{equation}
	\cfun\hom\at{u_j, \dmn} 
		= \intvol{\dmn}{\cden\hom\at{\wstn u_j}}
		= \intvol{\dmn}{\cden\hom\at{\rho_j \ast \wstn u}}
		\leq \int_{\dmn}{\cden\hom\at{\wstn u}}
		= \cfun\hom\at{\dmn}.
\end{equation}
Finally we obtain
\begin{equation}
	\prfGamma\limsup_{h} \cfun_{\varepsilon_{\sigma\at{h}}}\at{u,\dmn} \leq \cfun\hom\at{u, \dmn},
\end{equation}
which is the desired result. \qed
\end{proof}

\section{Conclusions and perspectives}\label{s:conclusions}
In the present work we have discussed a homogenization result dealing with masonry structures constituted by an assemblage of blocks with interposed mortar joints. Departing point of the derivation has been regarding such structures as a periodic collection of disconnected sets (blocks) interacting through their boundaries (interfaces). Accordingly, we have considered a sequence of energy functionals (scaling with the size of the microgeometry) on the space of special functions of bounded deformation comprising (i) a linear elastic behavior in the blocks, (ii) a Barenblatt's cohesive contribution at interfaces and (iii) a unilateral condition on the strain across interfaces. Exploiting the notion of $\prfGamma$convergence, we have analyzed the asymptotic behavior of such energy functionals, thus obtaining a simple homogenization formula for the limit energy. We have investigated the behavior of the limit energy, highlighting its mathematical and mechanical main properties. Among them, we have in particular focused on the non-standard growth conditions under tension or compression. 

Various additional questions and several perspectives arise from the presented results. One is to generalize this approach to the case of a more complicated energy contribution on interfaces, for instance mimicking a plastic behavior or friction. Furthermore, most of these problems can be rephrased in a nonlinearly elastic framework. Finally, yet requiring to abandon many of the techniques exploiting the convexity assumption here adopted, an interesting direction of investigation might be the extension to the Griffith theory of brittle fracture.

\subsection*{Acknowledgments}
The authors acknowledge the MIUR Excellence Department Project awarded to the Department of Mathematics, University of Rome Tor Vergata, CUP E83C18000100006.

\end{document}